\documentclass[10pt,a4paper,twoside]{amsart}

\usepackage{geometry}
\usepackage{etex}

\usepackage{amsfonts, amsmath,amssymb}
\usepackage[lite]{amsrefs}
\usepackage{amsthm}
\usepackage{tikz}
\usetikzlibrary{matrix,arrows}
\usepackage{subcaption}
\usepackage{color}
\usepackage{changepage}
\usepackage{hyperref}
\usepackage[pagewise]{lineno}

             \newcommand{\edgegraph}{\includegraphics[width=3.78mm, height=1.8mm]{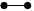}}
             \newcommand{\circlegraph}{\includegraphics[width=3.78mm, height=2.79mm]{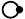}}

             \newcommand{\thetagraph}{\includegraphics[width=3.78mm, height=2.79mm]{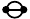}}
             \newcommand{\dumbbellgraph}{\includegraphics[width=7.2mm, height=2.79mm]{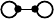}}

\newcommand{\hy}{\mathcal{H}}

\newcommand{\xsp}{X_s^\prime}
\newcommand{\hyper}{\mathcal{H}}
\newcommand{\SLO}{\mathrm{SL}_2\left(\mathcal{O}_{-m}\right)}
\newcommand{\PSLO}{\mathrm{PSL}_2\left(\mathcal{O}_{-m}\right)}

\newcommand{\SLtwo}{\mathrm{SL}_2}
\newcommand{\PSLtwo}{\mathrm{PSL}_2}
\newcommand{\go}{\text{P}{\Gamma_0}}
\newcommand{\calO}{{\mathcal O}}
\newcommand{\pslo}{\mathrm{PSL}_2(\calO_{-11})}
\newcommand{\pslod}{\mathrm{PSL}_2(\calO_{-m})}
\newcommand{\bC}{{\mathbb{C}}}
\newcommand{\Z}{{\mathbb{Z}}}
\newcommand{\N}{{\mathbb{N}}}
\newcommand{\ringOm}{\mathcal{O}_{-m}}
\newcommand{\rationals}{{\mathbb{Q}}}

\newcommand{\F}{{\mathbb{F}}}

\newcommand*{\Homol}{H}
\newcommand*{\Cohomol}{H}

\newcommand{\sign}{{\rm sign}}

\newcommand{\Kleinfourgroup}{\mathcal{D}_2}
\newcommand{\Di}{{\bf Di}}
\newcommand{\Q}{{\bf Q}_8}
\newcommand{\Te}{{\bf Te}}

\newcommand{\mat}{\begin{pmatrix}a & b\\c & d\end{pmatrix}}

\newcommand{\dE}{d_2^{\rm ESS}}
\newcommand{\dL}{d^{\rm LH3S}}
\newcommand{\calC}{\mathcal{C}}
\newcommand{\calG}{\mathcal{G}}

\theoremstyle{plain}
\newtheorem{thm}{\bfseries Theorem}
\newtheorem{theorem}[thm]{\bfseries Theorem}
\newtheorem{lemma}[thm]{\bfseries Lemma}
\newtheorem{proposition}[thm]{\bfseries Proposition}
\theoremstyle{remark}

\newtheorem{df}[thm]{\bfseries Definition}

\theoremstyle{plain}

\newtheorem{prop}[thm]{Proposition}

\newtheorem{corollary}[thm]{Corollary}
\theoremstyle{definition}

\newtheorem{rem}[thm]{Remark}

\newcommand{\ef}{{\mathbb F}_2}

\DeclareMathOperator{\rank}{rank}

\begin{document}

\title[Mod 2 cohomology of congruence subgroups in the Bianchi groups]{The mod 2 cohomology rings of \\congruence subgroups in the Bianchi groups}
\author[Ethan Berkove, Grant S. Lakeland and Alexander D. Rahm]{Ethan Berkove$^1$, Grant S. Lakeland$^2$ and Alexander D. Rahm$^3$
\\(with an appendix by Bui Anh Tuan and Sebastian Sch\"onnenbeck)}
 \address{Department of Mathematics, 230 Pardee Hall, Lafayette College, Easton, PA 18042, USA }
 \email{berkovee@lafayette.edu}
 \urladdr{http://sites.lafayette.edu/berkovee/}
 \address{Department of Mathematics and Computer Science,
 Eastern Illinois University,
 600 Lincoln Avenue,
 Charleston, IL 61920, USA}
 \email{gslakeland@eiu.edu}
 \urladdr{http://www.ux1.eiu.edu/~gslakeland/}
 \address{Universit\'e de la Polyn\'esie Fran\c{c}aise, Laboratoire de math\'ematiques GAATI, BP 6570, 98702 Faaa, French Polynesia}
 \email{Alexander.Rahm@upf.pf}
 \urladdr{http://gaati.org/rahm/}
\date{\today}
\subjclass[2000]{11F75, Cohomology of arithmetic groups.\\
1) Lafayette College, \url{https://math.lafayette.edu/people/ethan-berkove/}\\
2) Eastern Illinois University, \url{http://www.ux1.eiu.edu/~gslakeland/}\\
3) Universit\'e de la Polyn\'esie Fran\c{c}aise,  \url{http://gaati.org/rahm/}}

\begin{abstract}
We establish a dimension formula involving a number of parameters for the mod 2 cohomology of finite index subgroups in the Bianchi groups (SL$_2$ groups over the ring of integers in an imaginary quadratic number field).   The proof of our formula involves an analysis of the equivariant spectral sequence, combined with torsion subcomplex reduction.  We also provide an algorithm to compute a Ford domain for congruence subgroups in the Bianchi groups, from which the parameters in our formula can be explicitly computed.  

\end{abstract}

\maketitle

\section{Introduction}
Calegari and Venkatesh have recently proven a numerical form of a Jacquet--Langlands correspondence for torsion classes on arithmetic hyperbolic 3-manifolds~\cite{CalegariVenkatesh}.
This can be seen as a new kind of Langlands programme, for which one has to study torsion in the cohomology of arithmetic groups.
A class of arithmetic groups that is of natural interest here, as well as in the classical Langlands programme, consists of the congruence subgroups in the Bianchi groups.
By a \textit{Bianchi group}, we mean an SL$_2$ group over the ring of integers in an imaginary quadratic number field.
Our aim in this paper is to provide new tools for computing the torsion in the cohomology of the congruence subgroups in the Bianchi groups.

There are already several approaches known for studying congruence subgroups and their cohomology:
\begin{itemize}
\item Grunewald's method of taking a presentation for the whole Bianchi group, and deriving presentations for finite index subgroups via the Reidemeister-Schreier algorithm~\cites{GrunewaldSchwermer};
\item Utilizing the Eckmann--Shapiro lemma for computing cohomology of congruence subgroups directly from cohomological data of the full Bianchi group~\cite{RahmTsaknias};
\item Construction of a Vorono\"i cell complex for the congruence subgroup~\cites{Schoennenbeck, BCNS}.
\end{itemize}
What one typically harvests with these approaches are tables of machine results in which everything looks somewhat ad hoc.
A question posed by \mbox{Fritz Grunewald} to the third author asks for deeper structure.  Specifically, can one determine how the cohomology with small prime field coefficients develops when varying among Bianchi groups and their finite index subgroups,
and considering all cohomological degrees at once?
Our response to this question, in Theorem \ref{mainthm} below, is that far from being ad hoc, the cohomology of congruence subgroups of Bianchi groups depends for the most part on a surprisingly small amount of geometric data.  
This theorem also closes the $2$-torsion gap that could not be addressed by results already in the literature~\cite{AccessingFarrell}.  
(The pertinent formulas could not be applied to the case of $2$-torsion for the SL$_2$ subgroups because the $-1$ matrix is an omnipresent cell stabilizer of order $2$.)  
Let $\Gamma$ be a finite index subgroup in a Bianchi group not over the Gaussian integers, 
and let $X_s$ be the non-central $2$-torsion subcomplex of the action of $\Gamma$ on hyperbolic $3$-space $X$.
Note that technically we use a $2$-dimensional $\Gamma$-equivariant retract for $X$, 
but this does not affect the following formulas; in particular because $\Gamma$ is not over the Gaussian integers, 
it does not affect $X_s$. 
\begin{theorem} \label{mainthm}
Denote by $k \geq 0$ the number of conjugacy classes of subgroups of type $\Z/4$ in $\Gamma$ which are not included in any subgroup of type the quaternion group with $8$ elements ($\Q$) 
or the binary tetrahedral group ($\Te$) in $\Gamma$.
Let $m$, respectively $n$, be the number of conjugacy classes of subgroups of type $\Q$, respectively $\Te$, in $\Gamma$, 
and assume that $m > 0$, or $n > 0$, or both. 
Then \\
\begin{center}
$\dim_{\ef} \Cohomol^q(\Gamma;\thinspace \ef) = $\small$
\begin{cases}
 \beta^1 +m +k -r^{0,1}, & q = 1,\\
 \beta^2 +\beta^1 +2m+n +k -1 +c -r^{0,1} -r^{0,2}, & q \equiv 2 \mod 4,\\
  \beta^2 +\beta^1 +2(m+n) +k -1 +c -r^{0,3} -r^{0,2}, & q \equiv 3 \mod 4,\\
  \beta^2 +\beta^1 + m+n +k +c -r^{0,3}, & q \equiv 4 \mod 4, \\
    \beta^2 +\beta^1 +m +k +c -r^{0,1}, & q \equiv 5 \mod 4,\\
\end{cases}$\normalsize
\end{center}
where $\beta^q = \dim_{\F_2}\Cohomol^q(_\Gamma \backslash X ; \thinspace \F_2)$ for $q = 1, 2$;
$r^{0,q}$ is the rank of the $d_2^{0,q}$-differential of the equivariant spectral sequence;
and $c$ is the rank of the cokernel of the map
$\Cohomol^{1} (_\Gamma \backslash X; \thinspace \F_2) \rightarrow \Cohomol^{1} (_\Gamma \backslash X_s; \thinspace \F_2) $
induced by the inclusion $X_s \subset X$.
\\
We have the following vanishing results for these ranks:
\begin{itemize}
 \item $k = 0 \Rightarrow r^{0,3} = 0$,
 \item $c = 0 \Rightarrow r^{0,2} = 0$.
\end{itemize}

\end{theorem}
Note that this is a description of the cohomology in all degrees (we start at $q=1$), 
with the trade-off that the Betti numbers of the quotient space,
as a purely topological ingredient, remain an input.
In order to prove this theorem, which we do in Section~\ref{sec: d2 vanishes}, 
we describe an approach based on non-central torsion subcomplex reduction \cite{BerkoveRahm}, which 
is especially useful for computing the small torsion in the cohomology of our congruence subgroups.
The data we need for evaluating the formula of the above theorem comes from a fundamental domain for the action of the groups we are investigating on hyperbolic $3$-space.
Fundamental domains in hyperbolic $3$-space for arbitrary arithmetic Kleinian groups can be computed using the algorithm of Aurel Page~\cite{Page}. 
However, for the extraction of torsion subcomplexes, we need the fundamental domains to be provided with a cell structure in 
which every cell stabilizer fixes its cell pointwise, a condition not  produced by Page's current implementation (version 1.0).
Mathematically, it is straightforward to construct from Page's fundamental domain a subdivided one on which all cells are fixed pointwise.
Unfortunately, implementation costs for that approach are very high. For our purposes, it was more useful to start
from scratch with a new algorithm.

We present this algorithm in Section~\ref{sec:Ford}. 
The algorithm constructs a fundamental domain for a congruence subgroup in a  $2$-dimensional equivariant retract of hyperbolic $3$-space by 
starting with a fundamental domain for the ambient full Bianchi group which already has the desired cell structure. 
Then the desired property of the cell structure is inherited by the fundamental domain that we are constructing for the congruence subgroup. 
This allows us to efficiently extract torsion subcomplexes as well as to determine the number and type of connected component types---we use these data to evaluate the formulas of Theorem~\ref{mainthm}.

We can furthermore say which types of connected components can appear in $2$-torsion subcomplexes for subgroups of Bianchi groups---we give a complete characterization in Corollary~\ref{cor:Kramer}, 
which we derive from new results of Norbert Kr\"amer \cite{Kraemer}.
They imply in particular that in the notation of Theorem~\ref{mainthm},
a single connected component not of type~$\circlegraph$ admits either $m = 0, n = 2$ or 
$m = 2, n = 0$.  

For the proof of Theorem~\ref{mainthm}, we extend the tools developed in~\cite{BerkoveRahm} in order to reduce these torsion subcomplexes and to analyze the equivariant spectral sequence converging to group cohomology. 
By analyzing the remaining differentials in the equivariant spectral sequence (the lemmas in Section~\ref{sec: d2 vanishes}), we are able to determine almost all of the mod-$2$ cohomology of the congruence subgroups we consider, and can often get a complete answer.

A remark on the above mentioned Vorono\"i cell complex approach: recently we were able to address the non-triviality of the action of cell stabilizers on their cells via machine calculations using the Rigid Facets Subdivision algorithm~\cite{BuiRahm}.
This allows us, in the appendix to this paper, to use Sch\"onnenbeck's computations of the Vorono\"i cell complex
as a check on the computations in this paper, and to illustrate which values the parameters in our formulas can take.

\subsection*{Organization of the paper}
In Section~\ref{sec: background}, 
we recall background material.   In addition to some results from spectral sequences, this includes the definition of the non-central $\ell$-torsion subcomplex, 
whose components provide an efficient way to calculate cohomology rings for the groups we consider.  
In Section~\ref{sec:components}, we derive from recent results of Kr\"amer a categorization of all possible non-central $\ell$-torsion subcomplexes
in congruence subgroups of Bianchi groups for all occurring prime numbers~$\ell$.  In Section~\ref{Betti formula} we provide new theorems which calculate the cohomology of all possible reduced $2$-torsion subcomplexes.  
In Section~\ref{sec: d2 vanishes} we show that the $d_2$ differential vanishes in many cases in our setting, and provide the proof of Theorem~\ref{mainthm}. 
Section~\ref{sec:Ford} presents our algorithm with which to construct fundamental domains for the action of the congruence subgroups.  
Finally, Section~\ref{Example computations} provides two example computations.

\subsection*{Acknowledgments} 
The authors are grateful to the ``Groups in Galway'' conference series; the plans for this paper were made during a meeting which was co-organized by the third author, and attended by the second author and the appendix's second author.  The second author thanks Alan Reid for introducing him to this topic, and for numerous helpful conversations.
The third is thankful for being supported by Gabor Wiese's University of Luxembourg grant AMFOR.
Special thanks go to Norbert Kr\"amer for helpful suggestions on our manuscript.

\section{Preliminaries} \label{sec: background}

In this section we provide, without proof, background on subcomplex reduction and some spectral sequence results.  We refer the interested reader to the appropriate references for more details.  In the sequel, most of our cohomology calculations have coefficients in the field $\ef$ with two elements (obviously with the trivial action), so we will assume all cohomology is with $\ef$ coefficients unless explicitly stated otherwise.

\subsection{Subcomplex reduction}  \label{The non-central torsion subcomplex} 
Let $\Gamma$ be any discrete group which acts on a finite-dimensional simplicial complex $X$ 
via a cellular $\Gamma$--action so that elements of cell stabilizers fix their cells point-wise.  What we have in mind here is for $\Gamma$ 
to be a congruence subgroup in a Bianchi group which acts in the well-known way on hyperbolic $3$-space.
For a fixed prime number $\ell$, the $\ell$--\emph{torsion subcomplex} is the collection of all cells of $X$ whose cell stabilizer contains some element of order $\ell$.
If, in addition, for every non-trivial finite ($\ell$-)group $G \subseteq \Gamma$ the fixed point set $X^G$ is acyclic, 
we have the following special case of Brown's proposition X.(7.2) in~\cite{Brown}:
\begin{proposition} \label{Brownian}
  There is an isomorphism between the $\ell$--primary parts of the Farrell cohomology of~$\Gamma$ and the
  $\Gamma$--equivariant Farrell cohomology of the $\ell$--torsion subcomplex.
\end{proposition}

There are instances where the $\ell$--torsion subcomplex can be significantly reduced in size while preserving cohomological information.  
Heuristically, one determines situations, 
like adjacency conditions and isomorphisms on cohomology between stabilizer groups,
which allow cells to be merged.  
The result, which may have multiple components, is a reduced $\ell$--torsion subcomplex, 
and it can be significantly easier to work with than $X$ itself.  
Ad hoc use of elements of this approach appears in a number of places in the literature (see~\cite{Henn} or~\cite{Soule}, 
for example).  A systematic reduction procedure was developed by the third author in~\cites{RahmTorsion, RahmNoteAuxCRAS, AccessingFarrell}.  
One of the main results from~\cite{AccessingFarrell} is:
\begin{theorem} \label{pivotal}
  There is an isomorphism between the $\ell$--primary parts of the Farrell cohomology of~$\Gamma$ and the
  $\Gamma$--equivariant Farrell cohomology of a \emph{reduced} $\ell$--torsion subcomplex.
\end{theorem}

We note that Farrell cohomology coincides above the virtual cohomological dimension with the standard cohomology of groups,
and by the periodicity that we have for the family of groups treated in the present paper, 
we can deduce almost all 
of the cohomology of groups from the Farrell cohomology (i.e., except for the $1$-dimensional part).
This is why, although we make no mention of Farrell cohomology beyond this section, Proposition \ref{Brownian} and Theorem \ref{pivotal}
allow us to make a significant reduction in our computations on the cohomology of groups in this work. 
In cases where the action of $\Gamma$ on $X$ has a trivial kernel, one can use a reduced $\ell$--torsion subcomplex to determine the Farrell 
cohomology of~$\Gamma$~\cite{AccessingFarrell}.   On the other hand, 
when the kernel contains $\ell$--torsion, then the $\ell$--torsion subcomplex is all of $X$, yielding no reduction at all. The following way around this difficulty was developed in~\cite{BerkoveRahm}.

\begin{df} \label{non-central torsion subcomplex}
The \emph{non-central $\ell$--torsion subcomplex} of a $\Gamma$--cell complex $X$ is the union of the cells of $X$ 
whose cell stabilizer in~$\Gamma$ contains some element of order $\ell$ which is not in the center of~$\Gamma$.
\end{df}

We specialize to the case of our paper.  Let $\ringOm$ be the ring of integers in the field $\rationals(\sqrt{-m})$, for $m \in \N$ square-free. 
For a subgroup $\Gamma$ of $\SLO$, let $\text{P}\Gamma$ be the central quotient group $\Gamma / (\{1, -1\} \cap  \Gamma )$ in $\PSLO$.
As noted in~\cite{BerkoveRahm}, the \textit{non-central} $\ell$--torsion subcomplex for $\Gamma$ is the same as the $\ell$--torsion subcomplex for $\text{P}\Gamma$.  
In the sequel, we use this correspondence to identify the \textit{non-central} $\ell$--torsion subcomplex for the action of a congruence subgroup of $\SLO$ on hyperbolic $3$--space as the $\ell$--torsion subcomplex of the image of the subgroup in $\PSLO$.

It is well-known~\cite{SchwermerVogtmann} that the finite orders which can occur for elements of $\SLO$ are 1, 2, 3, 4 or 6; so in $\PSLO$ the finite orders are 1, 2 and 3.
From this, one can completely determine the finite subgroups of the $\PSLtwo$ Bianchi groups (see~\cite{binaereFormenMathAnn9}, for example), 
as well as their preimages in the $\SLtwo$ Bianchi groups: the cyclic groups $\Z/2$, $\Z/4$, $\Z/6$,  
the quaternion group $\Q$ of order $8$, the dicyclic group $\Di$ of order $12$, and the binary tetrahedral group $\Te$ of order $24$.  
Hence we only need to analyze the $2$--torsion subcomplex and the $3$--torsion subcomplex.

To characterize the latter $\ell$--torsion subcomplexes, recall that any element of finite order in $\PSLO$ fixing a point inside hyperbolic $3$-space~$\hyper$ acts as a rotation of finite order.
By Felix Klein's work, we also know that any torsion element~$\alpha$ is elliptic and hence fixes some geodesic line.
So our $\ell$--torsion subcomplexes are one-dimensional and consist of rotation axes of elements of finite order.
We can now deduce that for $\ell = 2$ or $3$, the $\Gamma$-quotient of the $\ell$--torsion subcomplex is a finite graph (we use either the finiteness of a fundamental domain, or a study of conjugacy classes of finite subgroups as in~\cite{Kraemer}). 

For the $\Gamma$-quotient of the $3$--torsion subcomplex, Theorem~2 of~\cite{AccessingFarrell} tells us that there are only two types of connected components, and gives the number of each.  Hence the mod $3$ cohomology groups of $\Gamma$ are explicitly expressed in terms of conjugacy classes of order-$3$-rotation subgroups and triangle subgroups in $\Gamma$.
The reduced non-central $2$--torsion subcomplex  for a congruence subgroup $\Gamma$ in~$\SLO$ is considerably more complicated, and its explicit determination is an emphasis of this paper. To start, it has the following properties:
\begin{itemize}
  \item All  edges of the non-central $2$--torsion subcomplex have stabilizer type $\Z/4$. 
  \item The stabilizer type of the vertices that yield an end-point in the quotient graph is the binary tetrahedral group $\Te$.
  \item  After carrying out the reduction, there are no vertices with precisely two neighboring vertices.  
  \item The stabilizer type of the vertices where there is a bifurcation in the quotient graph is the quaternion group $\Q$ with eight elements.
  \item  Prior to reduction, there are no vertices with more than three neighboring vertices in the quotient graph.
\end{itemize}
We note that the degree of each vertex in the $\Gamma$-quotient of $X$ is the same as the number of distinct conjugacy classes of $\Z/4$ in the vertex stabilizer.
We also observe that all stabilizers which contain a copy of $\Q$ are associated to vertices.
This observation will be used in Section~\ref{Section:E2 page}.

\subsection{Spectral sequences}\label{ssec: specseq}
We use two spectral sequences in this paper.  The primary one is the equivariant spectral sequence in cohomology, since it is particularly well-suited to our situation, that is, a cellular action of $\Gamma$ on the contractible cell complex $X$.  This spectral sequence is developed in detail in~\cite{Brown}*{Chapter VII}.  
For our purposes, we note that the $E_1$ page of the spectral sequence has the form 
\[
E_1^{i,j} \cong \prod\limits_{\sigma \thinspace\in \thinspace _\Gamma \backslash X^{(i)}} \Homol^j(\Gamma_\sigma)
\]
and converges to the cohomology of the group, $\Cohomol^{i+j}(\Gamma)$, where $X^{(i)}$ is a set of $\Gamma$-representatives of $i$--cells in $X$, and $\Gamma_\sigma \subseteq \Gamma$ is the stabilizer of the cell $\sigma$.  The $d_r$ differential of this spectral sequence has bidegree $(r,1-r)$.

We summarize a number of useful properties of the equivariant spectral sequence:
\begin{enumerate}
\item\label{ss1} The differential $d_1$ can be described (with $\ef$ coefficients) as the sum of restriction maps in cohomology between cell stabilizers (cohomology 
analog of~\cite{Brown}*{VII.8})
\[
\prod_{\sigma \thinspace\in \thinspace _\Gamma\backslash X^{(i)}} \Homol^j (\Gamma_\sigma)
\xrightarrow{\ d_1^{i,j}}\ 
\prod_{\tau \thinspace\in \thinspace _\Gamma\backslash X^{(i+1)}} \Homol^j(\Gamma_\tau).
\]
\item\label{ss2} There is a product on the spectral sequence, $E_r^{pq}  \otimes E_r^{st} \rightarrow E_r^{p+s,q+t}$, which is compatible with the standard cup product on $\Cohomol^*(\Gamma)$ ~\cite{Brown}*{VII.5}.
\item\label{ss3} On the $E_2$--page,
the products in  $E_2^{0,*}$, the vertical edge, can be identified with the restriction of the usual cup product to the classes in ${\prod}_ {\sigma \in _\Gamma \backslash X^{(0)} }   \Cohomol^q(\Gamma_{\sigma})$ which are in $\ker d_1$ ~\cite{Brown}*{X.4.5.vi}.
 \item  Let $r \geq 2$.  Then there are well-defined Steenrod operations, $Sq^k: E_r^{p,q} \rightarrow E_r^{p,q+k}$ when $0 \leq k \leq q$ \mbox{\cite{Singer}*{Theorem 2.15}.}
 \item\label{ss4}  When $k \leq q - 1$, $d_2 Sq^k u = Sq^k d_2 u$, i.e., the Steenrod operations commute with the differential ~\cite{Singer}*{Theorem 2.17}.
\end{enumerate}

In Sections \ref{Betti formula} and \ref{sec: d2 vanishes}, we also use the Lyndon--Hochschild--Serre spectral sequence in cohomology associated to the extension 
\begin{equation}
1 \rightarrow H \rightarrow \Gamma \buildrel \pi \over \longrightarrow  \Gamma/H \rightarrow 1  \label{sss:shortexactseq}.
\end{equation}
This short exact sequence yields an associated fibration of classifying spaces.  The Serre spectral sequence for this fibration 
has $E_2^{i,j} \cong \Homol^i(\Gamma/H; \Homol^j(H;M))$  for untwisted coefficients $M$ and converges to $\Homol^{i+j}(\Gamma;M)$.   For the development of this spectral sequence, we refer readers to either ~\cite{AdemMilgram}*{IV.1} or ~\cite{Brown}*{VII.5}.  A more general version for twisted coefficients can be found in \cite{MS}.

\subsection{Calculation of the second page of the spectral sequence} \label{Section:E2 page} 

One can use the reduced non-central $2$--torsion subcomplex of $\Gamma$ as input into the equivariant spectral sequence.  It is  one of 
the main results in~\cite{BerkoveRahm} that the $E_2$ page of the spectral sequence can be computed knowing limited data about the action
of the congruence subgroup $\Gamma$ via linear fractional transformations on hyperbolic $3$-space, $\mathbb{H}^3$.  
Let $X$ be the $2$-dimensional cellular retraction of $\mathbb{H}^3$ from
which we build the Ford domain for the action of $\Gamma$ (see~\cite{RahmFuchs}).  We denote by $X_s$ the 
non-central $2$-torsion subcomplex of $X$, and by $X_s^\prime$ the $0$-dimensional subcomplex of $X$ consisting of vertices
whose stabilizer is either $\Q$ or $\Te$.   We define $c$ to be the rank of the cokernel of 
$\Cohomol^{1} (_\Gamma \backslash X; \thinspace \F_2) \rightarrow \Cohomol^{1} (_\Gamma \backslash X_s; \thinspace \F_2) $
induced by the inclusion $X_s \subset X$.    That is, the case $c > 0$ corresponds to the situation where loops in the quotient of the non-central $2$-torsion subcomplex are identified in the quotient space.
We also define $v$ to be the number of conjugacy classes of subgroups isomorphic to $\Q$ in $\Gamma$; these occur in vertex
stabilizers isomorphic to $\Q$ or $\Te$.   There is a geometric meaning for $v$, related to the $2$-torsion subcomplex components 
found in $X$ as listed in Section~\ref{sec:components}.  Specifically, it is shown in~\cite{BerkoveRahm} that $v$ counts the number of vertices
in the non-$\circlegraph$ $2$-torsion subcomplex components of $X$.  Finally, we define $\sign(v)$ to equal $0$ when $v = 0$ and 
$1$ otherwise.  In the following theorem from~\cite{BerkoveRahm}, we add the trivial case when $X_s$ is empty.  That situation
did not arise in that paper, but it is straightforward to work it out.

 \begin{theorem}\cite{BerkoveRahm}*{Corollary 21}\label{thm:E2page}
 The $E_2$ page of the equivariant spectral sequence with $\F_2$--coefficients
 associated to the action of $\Gamma$ on $X$ is concentrated in the columns $p \in \{0, 1, 2\}$.
 If $X_s$ is empty, then $E_2^{p,q} \cong \Cohomol^p(_\Gamma \backslash X; \thinspace \F_2)$ for all $p, q \geq 0$;
otherwise this $E_2$ page has the following form:  
\[  
\begin{array}{l | cccl}
q \equiv 3 \mod 4 &  E_2^{0,3}(X_s)     &   E_2^{1,3}(X_s) \oplus (\ef)^{a_1}  &   (\ef)^{a_2} \\
q \equiv 2 \mod 4 &  \Cohomol^2_\Gamma(\xsp) \oplus (\ef)^{1 -\sign(v)}     &    (\ef)^{a_3}&   \Cohomol^2(_\Gamma \backslash X)  \\
q \equiv 1 \mod 4 &  E_2^{0,1}(X_s)     &  E_2^{1,1}(X_s) \oplus  (\ef)^{a_1} &   (\ef)^{a_2} \\
q \equiv 0 \mod 4   & \F_2  &   \Cohomol^1(_\Gamma \backslash X) &  \Cohomol^2(_\Gamma \backslash X) \\
\hline  &  p = 0 & p = 1 &  p = 2
\end{array}
\]
with 
\[\begin{array}{ll}
a_1 & =  \chi(_\Gamma \backslash X_s) -1 +\beta^1(_\Gamma \backslash X) +c   \\
a_2 & =  \beta^{2} (_\Gamma \backslash X) +c \\
a_3 & =  \beta^{1} (_\Gamma \backslash X) +v -\sign(v)
\end{array}
\]   
where $\beta^q = \dim_{\F_2}\Cohomol^q(_\Gamma \backslash X ; \thinspace \F_2)$ for $q = 1, 2$; and $\chi(_\Gamma \backslash X_s)$  
 is the usual Euler characteristic of the orbit space of the $2$-torsion subcomplex ${}_\Gamma \backslash X_s$.  
\end{theorem}

Theorem \ref{thm:E2page} implies that $\Cohomol^*(\Gamma)$ is eventually periodic with period $4$.  We remark that $\dim_{\F_2} \Cohomol^2_\Gamma(\xsp)$ is twice the number of orbits of vertices of stabilizer type~$\Q$ (cf.~\cite{BerkoveRahm}*{note 30}).
Finally, we note: 

\begin{lemma}\cite{BerkoveRahm}*{Lemma 22} \label{splitting over components}
The terms $E_2^{p,q}(X_s)$ split into direct summands each with support on one connected component
of the quotient of the \emph{reduced} non-central $2$--torsion subcomplex (again denoted $X_s$). 
\end{lemma}

\section{Connected component types of the 2-torsion subcomplex}\label{sec:components}
At first glance, the number of possible connected component types satisfying the properties listed in Section \ref{sec: background} is countably infinite. 
But, in fact, there are only four types.  This reduction is made possible by using recent results of  Kr\"amer~\cite{Kraemer}, which we summarize next.

\begin{corollary}[to theorems of Kr\"amer]\label{cor:Kramer}
Let $\Gamma_0(\eta)$ be a congruence subgroup in $\SLtwo(\ringOm)$, 
where $\ringOm$ is the ring of integers in the imaginary quadratic field $\rationals(\sqrt{-m})$ ($m$ square-free)  
with discriminant $\Delta \neq -4$,
and subject to the congruence condition that the lower left entry is contained in the ideal $\eta \subsetneqq \ringOm$.
Let $t$ be the number of distinct prime divisors of $\Delta$.
We consider the action of $\Gamma_0(\eta)$ on the associated symmetric space, hyperbolic 3-space.
Then the reduced non-central 2-torsion subcomplex consists exclusively of connected components of types
$\circlegraph$, $\edgegraph$, $\thetagraph$ and $\dumbbellgraph$, with multiplicities as follows.
\end{corollary}
\begin{adjustwidth}{1.2cm}{}
\begin{itemize}
 \item[$\edgegraph$] For every connected component of type $\edgegraph$, 
 there must be two conjugacy classes of binary tetrahedral subgroups in $\Gamma_0(\eta)$.
 Such subgroups exist precisely when $\eta = \langle 2\rangle$, $m \equiv 3 \bmod 8$ 
and for all prime divisors $p$ of $\Delta$, $p \equiv 1$ or $3 \pmod 8$.
There are then precisely $2^{t-1}$ connected components of this type.

\item[$\thetagraph$ \& $\dumbbellgraph$] The existence of either type $\thetagraph$ and $\dumbbellgraph$ requires two conjugacy classes of maximal $\Q$ groups per connected component.
The congruence subgroup $\Gamma_0(\eta)$ contains maximal $\Q$ subgroups if and only if $m \not\equiv 3 \bmod 4$ and in addition,
\begin{itemize}
 \item[$\bullet$] either $\eta = \langle 2\rangle$ and \\
 $p \equiv 1 \bmod 4$ for all odd prime divisors $p$ of $\Delta$,
 \item[$\bullet$] or $\eta^2 = \langle 2\rangle$ and \\
  for all divisors $D \in \mathbb{N}$ of $\Delta$, we have
 $D \not\equiv 7 \bmod 8$.
\end{itemize}

\item[$\dumbbellgraph$] A sufficient condition for all maximal $\Q$ groups to sit on 
$\dumbbellgraph$-components
is that $\eta^2 =\langle 2\rangle$, $m \equiv 2 \bmod 4$ and there exist $x, y \in \Z$ with $x^2 -m y^2 = 2$. 
When this condition holds, there are $2^{t-1}$ connected components of type $\dumbbellgraph$.

\item[$\thetagraph$] Conversely, sufficient conditions for all maximal $\Q$ groups to sit on $\thetagraph$-components 
are that $\eta =\langle 2\rangle$; or that $m \not\equiv 2 \bmod 4$; or that $x^2 -m y^2 \ne 2$ for all $x, y \in \Z$.
The number of connected components of type $\thetagraph$ is then 
\begin{itemize}
 \item[$\bullet$] $2^{t-1}$, if $\eta =\langle 2\rangle$.
 \item[$\bullet$] $2^{t-1}$, if $\eta^2 =\langle 2\rangle$ and $p \equiv  1 \bmod 8$ for all odd prime divisors $p$ of~$\Delta$.
 \item[$\bullet$] $2^{t-2}$, if $\eta^2 =\langle 2\rangle$ and $p \equiv \pm 3 \bmod 8$ for some prime divisor $p$ of~$\Delta$.
\end{itemize}
Note that the existence conditions directly imply that $p \not\equiv 7 \bmod 8$.
\end{itemize}
\end{adjustwidth}

\smallskip 

\noindent The remaining conjugacy classes of cyclic groups of order $4$ (i.e., those not involved in the components mentioned above)
constitute $\circlegraph$ components.  

\begin{proof}[Proof of the corollary]
We first provide a sketch of the overall argument.  The theorems that we quote are results in Kr\"amer's preprint~\cite{Kraemer}, so we refer the reader there for more details.
Satz 9.7 in \cite{Kraemer} implies that every stabilizer group of binary tetrahedral type occurs only
as an endpoint of a component of type $\edgegraph$.
For maximal $\Q$ vertex stabilizer subgroups, the conditions for the existence of components of types $\thetagraph$ and $\dumbbellgraph$
are complementary, so these are the only components that can occur in this case. 
The remaining connected components admit exclusively $\Z/4$ vertex stabilizers,
so they are of type $\circlegraph$.      
\\
For the individual component types, we observe:
\begin{adjustwidth}{1.2cm}{}
 \begin{itemize}
  \item[$\edgegraph$]
Satz 9.4.(i) states the conditions given in this Corollary for the existence of binary tetrahedral groups.
Satz 9.6.(i) specifies the number of conjugacy classes of binary tetrahedral type as $2^t$.  
There are two such conjugacy classes needed for each $\edgegraph$ component.

\item[$\thetagraph$ \& $\dumbbellgraph$]
The proof of conditions for existence are given in Satz 9.4.(iii).

\item[$\dumbbellgraph$] The proof of the sufficient condition for $\dumbbellgraph$ 
is given with Satz 9.9.(i).  This condition yields, 
as is stated in Satz~9.9.(i), that for all odd prime divisors $p$ of $\Delta$, 
we have $p \equiv 1 \pmod 8$. 
Therefore by Satz 9.6.(iii), 
there are $2^{t}$ conjugacy classes of maximal $\Q$-groups.

\item[$\thetagraph$] The proof of the sufficient conditions for $\thetagraph$ is given in Satz 9.9.(ii).  The proof of the number of conjugacy classes is given in  Satz 9.6.(iii).
 \end{itemize}
\end{adjustwidth}
\end{proof}
\noindent
\textbf{Examples.}  We have explicitly computed the type and number of components for a number of fundamental domains computed with the algorithm in Section~\ref{sec:Ford}.
The following results for the non-central $s$-torsion subcomplex quotient,
$_{\Gamma} \backslash X_s$ for $s \in \{2, 3\}$, are in accordance with Corollary~\ref{cor:Kramer}.
$$\begin{array}{|c|c|c|c|} \hline
\text{Example} & \Gamma  &  _{\Gamma} \backslash X_3  &  _{\Gamma} \backslash X_2  \\
\hline 
(1)& \Gamma_0(1 +\sqrt{-2}) \subset \SLtwo(\Z[\sqrt{-2}]) & \circlegraph & \text{empty} \\  

(2) & \Gamma_0(2) \subset \SLtwo(\Z[\sqrt{-2}]) & \text{empty} & \thetagraph \\

(3) & \Gamma_0(5) \subset \SLtwo(\Z[\sqrt{-2}]) & \circlegraph \circlegraph & \circlegraph \circlegraph \\

(4) & \Gamma_0(\sqrt{-2}) \subset \SLtwo(\Z[\sqrt{-2}]) & \text{empty} & \dumbbellgraph \\

(5) & \Gamma_0(3 +2\sqrt{-2}) \subset \SLtwo(\Z[\sqrt{-2}]) & \text{empty} & \circlegraph \circlegraph \\

(6) & \Gamma_0(2) \subset \SLtwo(\Z[\frac{-1+\sqrt{-11}}{2}]) & \circlegraph \circlegraph & \edgegraph \\

(7) & \Gamma_0(\frac{-1+\sqrt{-11}}{2}) \subset \SLtwo(\Z[\frac{-1+\sqrt{-11}}{2}]) & \circlegraph & \text{empty} \\
\hline
\end{array}$$
Details for Examples (3) and (4) in the table above are given in Section~\ref{Example computations}.

\section{Cohomology of $2$-torsion subcomplexes}\label{Betti formula}

The cohomology of the $2$-torsion complexes is built from the cohomology of its finite stabilizers.  We get from the former to the latter via a spectral sequence argument, using a general description of the $2$-torsion subcomplex.   We begin with a list of the cohomology rings for the stabilizer groups, where polynomial classes are given in square brackets, exterior classes are given in parentheses, and a subscript on a class denotes its degree.   Since SU$_2$ acts freely on $S^3$, the rings are all periodic of period dividing $4$.  

\begin{prop} (See~\cite{AdemMilgram})  \label{prop:SLfincoh}      
The mod $2$ cohomology rings of the finite subgroups of $\SLtwo(\ringOm)$ are:
\begin{alignat}{3}
  &  \Cohomol^*(\Z/4)  & \cong  &    \Cohomol^*(\Di)     \cong \ef[e_2](b_1)  \nonumber \\
  &  \Cohomol^*(\Z/2)  & \cong  &   \Cohomol^*(\Z/6)  \cong  \ef[e_1]  \nonumber \\
  &  \Cohomol^*(\Q)    & \cong  & \   \ef[e_4](x_1, y_1) / \langle R \rangle, \text{ with } R \text{ generated by }  x_1^2 + x_1 y_1+y_1^2 \text{ and } x_1^2 y_1 + x_1 y_1^2 \nonumber \\
  &  \Cohomol^*(\Te)   & \cong  &  \  \ef[e_4](b_3) \nonumber
\end{alignat}
\end{prop}
The calculation of the cohomology of the reduced $2$ torsion subcomplexes from a graph of groups description requires knowledge of restriction maps in cohomology between finite subgroups.  The proof of the following Proposition  can be found in~\cite{BerkoveRahm}*{Proposition 10}.
\begin{prop}\label{prop:restriction}
The following are the nontrivial restriction maps involving polynomial generators in cohomology for finite subgroups of $\SLtwo(\ringOm)$:
\[
\begin{array}{ll}
\Z/4: & res^{\Z/4}_{\Z/2} (e_2)  = e_1^2  \\
\Di:   & res^{\Di}_{\Z/6}(e_2)  = res^{\Di}_{\Z/2}(e_2) = e_1^2 \\
\Q:   & res^{\Q}_{\Z/4}(e_4)  =  e_2^2, \ res^{\Q}_{\Z/2}(e_4)  = e_1^4  \\
\Te:  & res^{\Te}_{\Z/6}(e_4)  = res^{\Te}_{\Z/2}(e_4) = e_1^4,  \ res^{\Te}_{\Z/4}(e_4)  = e_2^2 \\
\end{array}{}
\]
In addition, $res^{\Z/6}_{\Z/2}$ and $res^{\Di}_{\Z/4}$ are isomorphisms.
\end{prop}

The restriction map between $\Cohomol^1(\Q)$ and $\Cohomol^1(\Z/4)$ is trivial or not depending on the choice of $\Z/4$ subgroup.  This is the subject of the next lemma.

\begin{lemma}\label{lem:Q8res}
Given a class in $\Cohomol^1(\Q)$, its image under $res^{\Q}_{\Z/4}$ is non-trivial for two copies of $\Z/4 \subseteq \Q$ but trivial on the third.
\end{lemma}

\begin{proof}
Let $\Kleinfourgroup$ denote the Klein four-group.  There are three elements of order $2$ in $\Kleinfourgroup$, and their central extensions become the   
three copies of $\Z/4$ in $\Q$.  Fix $x$ and $y$, two generators of $\Kleinfourgroup$, viewed multiplicatively, 
and let $x_1$ and $y_1 \in \Cohomol^1(\Kleinfourgroup)$
be their corresponding duals in cohomology.  We identify the three subgroups of order $2$ in $\Kleinfourgroup$ as
$Z_1 = \langle x \rangle$, $Z_2 = \langle y \rangle$, and $Z_3 = \langle xy \rangle$, with corresponding
cohomology generators $z_{1,1}$, $z_{1,2}$, and $z_{1,3}$.

Determination of most of the restriction maps $res^{\Kleinfourgroup}_{Z_i}$ is straightforward, but as $xy$ is the product of the distinguished generators, in 
$res^{\Kleinfourgroup}_{Z_3}$ both $x_1$ and $y_1$ are sent to $z_{1,3}$.  So
\[
 \begin{array}{ll}
    res^{\Kleinfourgroup}_{Z_1}(x_1) = z_{1,1} & res^{\Kleinfourgroup}_{Z_1}(y_1) = 0       \\
    res^{\Kleinfourgroup}_{Z_2}(x_1) = 0      & res^{\Kleinfourgroup}_{Z_2}(y_1) = z_{1,2}  \\    
    res^{\Kleinfourgroup}_{Z_3}(x_1) = z_{1,3} & res^{\Kleinfourgroup}_{Z_3}(y_1) = z_{1,3}  \\
\end{array}
\]  
We determine $res^{\Kleinfourgroup}_{Z_i}(x_1 + y_1)$ by addition, and the result for $\Q$ follows from a comparison of the Lyndon--Hochschild--Serre spectral sequence associated to 
\[
1 \rightarrow \Z/2 \rightarrow \Q \rightarrow \Kleinfourgroup \rightarrow 1
\]
with the Lyndon--Hochschild--Serre spectral sequence for 
\[
1 \rightarrow \Z/2 \rightarrow \Z/4 \rightarrow \Z/2 \rightarrow 1.
\]
\end{proof}

We calculate the $\Gamma$-equivariant mod $2$ cohomology of a generic component of the non-central reduced $2$-torsion subcomplex 
via a graph of groups description and the equivariant spectral sequence.  
We use Proposition~\ref{prop:restriction} and Lemma~\ref{lem:Q8res} to determine the $d_1$ differential.  
By the cohomology periodicity of the stabilizer subgroups, it is sufficient to restrict ourselves to dimensions $q \leq 4$.  
The restriction map between $\Cohomol^1(\Q)$ and $\Cohomol^1(\Z/4)$ requires the most attention, 
since each $\Q$ vertex stabilizer contains three non-conjugate copies of $\Z/4$, 
a fact reflected in the three edges incident to that vertex.  
There are two cases, according as whether two of the incident edges form a loop or not.  
We call an edge which forms a loop in $\Gamma \backslash X_s$ a \textit{looped edge}.

\begin{lemma}\label{lem:Q8noloop}
Let $\Q$ be a vertex stabilizer with no looped edges.  
Then under $d_1$, any class in $\Cohomol^1(\Q)$ restricts isomorphically to exactly two copies of $\Cohomol^1(\Z/4)$ and is trivial on the third.
\end{lemma}

\begin{proof}
 This follows directly from Lemma~\ref{lem:Q8res}.  We note that the three classes in $\Cohomol^1(\Q)$ are detected on different pairs of $\Z/4$ subgroups.
\end{proof}

\begin{lemma}\label{lem:Q8yesloop}
Let $\Q$ be a vertex stabilizer adjacent to a looped edge, and   
let $b_{1,1}$ and $b_{1,2}$ be the classes in $\Cohomol^1(\Z/4)$ associated to the looped edge and unlooped edge respectively.  
Then under $d_1$, any class in $\Cohomol^1(\Q)$ is either detected by the cohomology of both edge stabilizers, or neither.  
\end{lemma}

\begin{proof}
Let $y$ be the group element of $\Gamma$ associated with the unlooped edge, 
and $x$ and $xy$ be the group elements of $\Gamma$ associated with one side of the looped edge.  
We again use Lemma~\ref{lem:Q8res}.  However, we note that the subgroup $Z_3$ is now identified with $Z_1$, 
so in cohomology $b_{1,1} = b_{1,3}$.  For the unlooped edge stabilizer, associated to $Z_2$, we have  
\[
 \begin{array}{ll}
   res^{\Q}_{\Z/4}(x_1) = 0  & res^{\Q}_{\Z/4}(y_1) = b_{1,2}  \\
 \end{array}
\]
The restriction to the looped edge stabilizer is induced by the difference of restriction maps to $Z_1$ and $Z_3$ in Lemma~\ref{lem:Q8res}.  
Therefore, this restriction map is given by
\[
 \begin{array}{ll}
   res^{\Q}_{\Z/4}(x_1) = b_{1,1} - b_{1,1} = 0  & res^{\Q}_{\Z/4}(y_1) = b_{1,1}  \\
 \end{array}
\]
We determine the restriction on $x_1 + y_1$ by additivity, and the result follows.
\end{proof}
It is usually clear by Proposition~\ref{prop:restriction} which classes are in the kernel of $d_1$ and which are in its image.
However, Lemmas~\ref{lem:Q8noloop} and~\ref{lem:Q8yesloop} show that the situation is more subtle for classes in degree~$1$.   
We introduce a graphical idea which will aid us in determining $\ker d_1$ 
in this case. First, note that the only classes in $E_2^{0,1}$ come from copies of $\Q$ vertex stabilizers.  Therefore,
any class in $E_2^{0,1}$ can be written as $\sigma \in \oplus_S \Cohomol^1(\Q)$, where the finite index set $S$ gives the support of $\sigma$.  By Property~\ref{ss1} in
Subsection~\ref{ssec: specseq}, the map $d_1$ is the difference of restriction maps for vertex stabilizer groups to edge stabilizer groups.  
We make the following observation about unlooped edges $e$ with endpoints $v_0$ and $v_1$: for a class to be in $\ker d_1$, 
a necessary condition (mod 2) is that the restriction maps to the edge stabilizers must both vanish, or both must be non-trivial.  
That is, for unlooped edges,
\begin{equation}\label{eqn:condition}
  res^{\Gamma_{v_1}}_{\Gamma_{e}}  =  res^{\Gamma_{v_0}}_{\Gamma_{e}}.
\end{equation}
We understand the restriction maps from $\Cohomol^1(\Q)$.  When there are no looped edges,  a class in $\Cohomol^1(\Q)$ is detected by the stabilizers of exactly two edges
by Lemma~\ref{lem:Q8noloop}; and when there is a looped edge, by Lemma~\ref{lem:Q8yesloop}, a class in $\Cohomol^1(\Q)$ is either detected by stabilizers on both the looped and unlooped edge, or it is sent to $0$.  

We remark that it is possible to order basis elements in such a way that $d_1$ is described by a block matrix where each block is associated to a single 
connected component $C$.  Therefore, we can analyze $d_1^{0,1}$ one component at a time.  We have our first result.

\begin{lemma} \label{lem:loopclass}
For each looped edge in a non-central $2$-torsion subcomplex quotient $C$ not of type $\circlegraph$, there is a class which is in $\ker d_1^{0,1}$.
\end{lemma}

\begin{proof}
The class in $\ker d_1^{0,1}$ is $x_1$ from Lemma~\ref{lem:Q8yesloop}.
\end{proof}

Given a cohomology class $\sigma \in \oplus_S \Cohomol^1(\Q)$ from a non-central $2$-torsion subcomplex quotient component $C$ not of type $\circlegraph$,
the \textit{support} of $\sigma$ is a subgraph of $C$ built as follows:  For each copy of $G \cong \Q$ which contributes to $\sigma$,
add to the subgraph its associated vertex, and the incident edges which detect $\sigma |_G$ as given by Lemmas~\ref{lem:Q8noloop} and~\ref{lem:Q8yesloop}.

\begin{lemma}\label{lem:leafless}
If $\sigma \in \oplus_S \Cohomol^1(\Q)$ has support which includes an edge incident to a loop in $C$, then $\sigma \not\in \ker d_1^{0,1}$.
\end{lemma}

\begin{proof}
We note that the only reduced non-central $2$-torsion subcomplex quotient components which have $\Z/4$ vertex stabilizers are $\circlegraph$ loops, 
so they do not fall under the hypothesis of this lemma.    Like all edge stabilizers in $C$, the stabilizer of $e$, $Stab(e)$, is isomorphic to $\Z/4$.   
We will focus on a $\Q$ vertex stabilizer on the other side of $e$, denoting this distinguished copy of $\Q$ by $Q$.  

Since $e$ is adjacent to a looped edge, by Lemma~\ref{lem:Q8yesloop}, the restriction map in cohomology from the vertex group adjacent to the looped edge either maps non-trivially to both edge groups or to neither.  If both, then there can be no second restriction map to the cohomology of the looped edge to make $d_1^{0,1}$ vanish.  If neither, then the support arises because $res^Q_{Stab(e)} \neq 0$.  
By Equation~(\ref{eqn:condition}), such a class cannot be in $\ker d_1^{0,1}$.    
\end{proof}

\begin{theorem}\label{thm:d1rank}
Given $C$, a reduced non-central $2$-torsion subcomplex quotient component not of type $\circlegraph$ , the block of $d_1^{0,1}$ supported on $C$ satisfies $\dim_{\F_2} (\ker d_1^{0,1} |_C) = \beta_1(C)$,
where $\beta_1(C)$ is the first Betti number of $C$.
\end{theorem}

\begin{proof}
Since the reduced non-central $2$-torsion subcomplex quotient components are disjoint, we can arrange classes in the equivariant spectral sequence so that cohomology calculations may be done one component at a time.   
	
For components of type $\thetagraph$, denote the two $\Q$ vertex stabilizers by $Q_1$ and $Q_2$, and the three edges by $e_1, e_2$, and $e_3$.  By Lemma \ref{lem:Q8res}, we can pick bases for $\Cohomol^1(Q_1)$ and $\Cohomol^1(Q_2)$ so that two classes restrict non-trivially to edges $e_1$ and $e_2$,   and the other two classes restrict non-trivially to edges $e_1$ and $e_3$.  Moving to the equivariant spectral sequence calculation for $\thetagraph$, the sum of the first two basis elements map to zero as do the sum of the second two.  Therefore, $\dim_{\F_2} \ker d_1^{0,1} \geq 2$.  As the classes in $\Cohomol^1(Q_1)$ are linearly independent, we conclude that the dimension is exactly two, which is $\beta_1(\thetagraph)$.

 For a component of type $\dumbbellgraph$, we recall that by Lemma ~\ref{lem:loopclass}, each looped edge results in exactly one class in $\ker d_1^{0,1}$, so $\dim_{\F_2} \ker d_1^{0,1} \geq 2$.  As above, denote the two copies of $\Q$ stabilizers by $Q_1$ and $Q_2$.  By Lemma \ref{lem:Q8res}, we can pick bases for $\Cohomol^1(Q_1)$ and $\Cohomol^1(Q_2)$ so that one class in each of $\Cohomol^1(Q_1)$ and $\Cohomol^1(Q_2)$ restricts non-trivially to a loop and to the edge connecting the vertices.  These classes are linearly independent, which implies $\dim_{\F_2} \ker d_1^{0,1} = \beta_1(\thetagraph) = 2$.
 
Finally, for a component of type $\edgegraph$ the result follows by Proposition \ref{prop:SLfincoh}, since $\Cohomol^1(\Te) = 0$.  
\end{proof}

In our non-central reduced $2$--torsion subcomplexes, 
edge stabilizers are always of type $\Z/4$; 
and vertex stabilizers are of one of the types 
\begin{itemize}
 \item $\Te$, with one edge adjacent to the vertex in the quotient space, or
 \item $\Z/4$, with one edge adjacent at both of its ends, yielding a connected component $\circlegraph$, or  
 \item $\Q$, with three edges (counted once or twice) adjacent to the vertex in the quotient space.
\end{itemize} 
Corollary~\ref{cor:Kramer} 
tells us that the only possible types of connected components in our quotients of reduced non-central $2$--torsion subcomplexes are
$\circlegraph$, $\edgegraph$, $\thetagraph$ and $\dumbbellgraph$.  
In the following theorem, we treat the components of types  $\edgegraph$, $\thetagraph$ and $\dumbbellgraph$ in a unified way,
making use of the fact that the only types of vertex stabilizers on these three connected component types are $\Q$ and $\Te$.

\begin{theorem}\label{thm:generalcohom}
Let $\calC$ be the collection of connected components not of type $\circlegraph$ 
in some reduced $2$-torsion subcomplex,
where $m$ vertices have $\Q$ stabilizers and $n$ vertices have $\Te$ stabilizers.
Then the $E_2$ page of the equivariant spectral sequence restricts on $\calC$ to
the following dimensions over $\ef$:
\begin{center}
$ \begin{array}{l | l l  l}
q \equiv 3 \mod 4 & {m+n}  &  {3\frac{m}{2}+\frac{n}{2}}  \\
q \equiv 2 \mod 4 & {2m}   &  {3\frac{m}{2}+\frac{n}{2}}  \\
q \equiv 1 \mod 4 & {m}    &  {\frac{m}{2}+\frac{n}{2}}  \\
q \equiv 0 \mod 4 & {\frac{m}{2}+\frac{n}{2}}        &  {m}  \\
\hline & p = 0 & p =  1
\end{array}
$ 
\end{center}
\end{theorem}

\begin{proof}
By Kr\"amer's results (see Corollary~\ref{cor:Kramer}), the collection
$\calC$ contains $3\frac{m}{2}+\frac{n}{2}$ edges, all of which have $\Z/4$ stabilizer.  
The $E_1$ page of the equivariant spectral sequence has the following dimensions over $\ef$.
$$
\begin{array}{l | l l l l}
q \equiv 3 \mod 4 & {m+n} &  \longrightarrow & {3\frac{m}{2}+\frac{n}{2}}  \\
q \equiv 2 \mod 4 & {2m}  &  \longrightarrow & {3\frac{m}{2}+\frac{n}{2}}  \\
q \equiv 1 \mod 4 & {2m}  &  \longrightarrow & {3\frac{m}{2}+\frac{n}{2}}  \\
q \equiv 0 \mod 4 & {m+n} &  \longrightarrow & {3\frac{m}{2}+\frac{n}{2}}  \\
\hline & p = 0 & & p = 1
\end{array}
$$
To get to the $E_2$ page, we need to determine the $d_1$ differential.  
In the bottom row, the $E_2^{p,0}$ term of $\calC$ is isomorphic to the simplicial homology $\Homol_p(\calG)$ of the graph $\calG$ underlying $\calC$.
The dimension of $\Homol_0(\calG)$ is the number of connected components of $\calG$, namely $\frac{m}{2}+\frac{n}{2}$.
And the dimension of $\Homol_1(\calG)$ is the number of loops of $\calG$; 
using Kr\"amer's results (see Corollary~\ref{cor:Kramer}) we see that there are as many loops in the collection $\calC$ as bifurcation points, namely $m$.

In dimensions $q \equiv 1 \bmod (4)$,  
Theorem~\ref{thm:d1rank} implies that $\dim \ker d_1^{0,q} = m$.  
And from Proposition~\ref{prop:restriction}, we know that  
$d_1^{0,q}$ vanishes in dimensions $q \equiv 2, 3 \bmod (4)$.  
\end{proof}

For the excluded case, we recall a lemma that has already been established.

\begin{lemma}[\cite{BerkoveRahm}*{Lemma 26}] \label{circle-contribution}
 Let $C$ be a connected component of type $\circlegraph$.  Then 
 
 $ \dim_{\ef}\Cohomol^q(C) = $\scriptsize$
  \begin{cases}
       1,             &    q = 0;\\
       2,     &    q  \geq 1
  \end{cases}
$ \normalsize
and $E_2^{p,q}|_C \cong \ef$ for all $q \geq 0$, $p \in \{0, 1\}$.
\end{lemma}
The proof of Lemma \ref{circle-contribution} is straightforward.  Briefly, at the level of graph of groups a connected component of type $\circlegraph$ is an HNN extension where the twisting sends a generator of the vertex stabilizer to another generator.  Property~\ref{ss1} in Section~\ref{ssec: specseq} then implies that $d_1$ is the zero map.

\section{Determination of the $d_2$ differential} \label{sec: d2 vanishes}
In order to complete our cohomology calculations, we need to understand the $d_2$ differential in the equivariant spectral sequence.
In general this can be quite difficult, but in our situation the existence of Steenrod operations in the spectral sequence help out 
a great deal.  We mirror the approach laid out in~\cite{BerkoveRahm}, quoting results as needed.
We treat first the degenerate case where the non-central $2$-torsion subcomplex $X_s$ is empty,
so we can afterwards assume that it is non-empty.
\begin{proposition} \label{empty case}
Let the non-central $2$-torsion subcomplex $X_s$ be empty.
Then $$\dim_{\ef} \Cohomol^q(\Gamma;\thinspace \ef) = 
\begin{cases}
 \beta^1 +1, & q = 1,\\
 \beta^2 +\beta^1 +1, & q \geq 2,
\end{cases}$$
 where $\beta^q = \dim_{\F_2}\Cohomol^q(_\Gamma \backslash X ; \thinspace \F_2)$.
\end{proposition}
\begin{proof}
 In view of Theorem~\ref{thm:E2page}, we only have to show that the $d_2$ differential vanishes completely.
 For this purpose, we consider the homological equivariant spectral sequence with integer coefficients.
 That sequence is, modulo $3$-torsion and apart from the zeroth row, concentrated in the odd rows,
 with $(p,q)$th entry $(\Z/2)^{\beta^p}$.
 Therefore its $d_2$ differential is zero modulo $3$-torsion in rows above $q = 1$.
 Comparing via the Universal Coefficient theorem, we see that the $d_2$ differential in rows above $q = 1$
 must vanish also for the cohomological equivariant spectral sequence with mod $2$ coefficients.
 By the periodicity, we obtain that the latter differential vanishes also in low degrees.
\end{proof}

Recall from Section~\ref{Section:E2 page} the definition of $c$ as  the rank of the cokernel of
\[
\Cohomol^{1} (_\Gamma \backslash X; \thinspace \F_2) \rightarrow \Cohomol^{1} (_\Gamma \backslash X_s; \thinspace \F_2) 
\]
induced by the inclusion $X_s \subset X$.  

\begin{lemma}\label{lem:d2evenvanish}
When $c= 0$, the $d_2$ differential vanishes on $E_2^{0,4q+2}$.
\end{lemma}

\begin{proof}
This is essentially Lemma 31 in~\cite{BerkoveRahm}, although a necessary hypothesis that $c = 0$ is missing there.  We sketch the argument, indicating the need for $c=0$ right after the proof.  

We note that $Sq^1$ vanishes on $\Cohomol^2(G)$ for all finite groups $G$ which appear as vertex stabilizers.  In the equivariant spectral sequence, a non-zero target of the $d_2$ differential is an odd-dimensional class in the cokernel of $d_1$.  Since all $2$-cells have $\Z/2$ stabilizer, the cohomology groups in $E_1^{2,q}$ are $\ef$-vector spaces.  Therefore, we may consider classes in $E_2^{2,q}$ to be equivalence classes of the form $\oplus \Cohomol^*(\Z/2)$, where the sum is over appropriate $2$-cells in $_\Gamma \backslash X$. (Note: in what follows, we will simply refer to sums on the $E_2$ page, although we mean equivalence classes.) Looking back at the derivation of Theorem~\ref{thm:E2page}, classes of this form have non-trivial $Sq^1$ when $c = 0$ since all classes in the second column arise from cohomology of stabilizers of $2$-cells in $_\Gamma \backslash X$ and the $\F_2$-dimensions of $E_2^{2,q}$  are equal for all $q$.  However, by Property~\ref{ss4} in Section~\ref{ssec: specseq}, $Sq^1 d_2 = d_2 Sq^1$, so this is impossible.  
\end{proof}
On the other hand, when $c > 0$, some classes in $E_2^{2,1}$ will have trivial $Sq^1$ simply for dimensional reasons. These classes can be the target of a non-trivial $d_2$ differential.

\begin{rem}
The $d_2^{0,2q}$ differential may vanish also in cases where $c \neq 0$.
In the latter cases, we rely on the machine computations in order to find 
the rank of $d_2^{0,2q}$.
\end{rem}

In addition, we have 

\begin{lemma}\label{lem:d24qVanish}
The $d_2$ differential vanishes on $E_2^{0,4q}$.
\end{lemma}

\begin{proof}
By periodicity, it is sufficient to show this in $E_2^{0,4}$.  If $d_2$ does not vanish, then its image is in $E_2^{2,3}$ which is generated by a class in $\oplus \Cohomol^3(\Z/2)$, where the sum is over $2$-cells in $_\Gamma \backslash X$.  We follow the argument in Lemma~\ref{lem:d2evenvanish}, noting that since $\dim_{\ef} E_2^{2,3} = \dim_{\ef} E_2^{2,5}$, $Sq^2$ of any class in $E_2^{2,3}$ is non-zero.  On the other hand, $Sq^2$ of the $4$-dimensional polynomial class is always $0$.  Now Property~\ref{ss4} in Section~\ref{ssec: specseq} implies $Sq^2 d_2 = d_2 Sq^2$, which forces the vanishing result.  
\end{proof}

The next few results relate to components of type $\circlegraph$.  

\begin{lemma}\cite{BerkoveRahm}*{Lemma 33}\label{lem:BR33}
The $d_2$ differential is nontrivial on cohomology on components of type $\circlegraph$ in degrees $q \equiv 1 \bmod 4$
if and only if it is nontrivial on these components in degrees $q \equiv 3 \bmod 4$.
\end{lemma}

\begin{lemma}\label{lem:circle2periodic}  
Let the non-central $2$-torsion subcomplex $X_s$ admit as  quotient components of type
$\circlegraph$ only.  Then the $d_2$ differential vanishes on $E_2^{0,2q}$.
\end{lemma}

\begin{proof}
 To prove this lemma, we will compare results from two different spectral sequences, the equivariant spectral sequence for the PSL$_2$ group (ESS), and the Lyndon-Hochschild-Serre spectral sequence (LH3S).  We will distinguish the differentials in the spectral sequences via their superscripts.  
 Let  $\Gamma \subseteq \SLO$ be a congruence subgroup and $\text{P}\Gamma$ be its image in $\PSLO$.  By the discussion after Definition~\ref{non-central torsion subcomplex}, the reduced non-central $2$-torsion subcomplex of $\Gamma$ and reduced $2$-torsion subcomplex of $\text{P}\Gamma$ are the same (although the cell stabilizers differ by a central $\Z/2$).  

Recall from Theorem~\ref{thm:E2page} that under the assumptions of this lemma, $\xsp$ is empty and $\sign(v) = 0$.  Therefore, it is sufficient to show that the 2-dimensional class in $E_2^{0,2}$ from Theorem~\ref{thm:E2page} is a permanent cocycle.  We note that this class is the unique polynomial generator in $\Cohomol^*(\Gamma)$.    Let $\Gamma$ be a congruence subgroup where $\text{P}\Gamma \subseteq \PSLO$ has a reduced $2$-torsion subcomplex which consists solely of $k$ components of type $\circlegraph$.    In $\PSLO$, components of type $\circlegraph$ correspond to a graph of groups where there is a single vertex and single edge, both with $\Z/2$ stabilizer.   
This, in turn, corresponds to the HNN extension $\langle t, x | x^ 2 = 1, t^{-1} x t = x \rangle$.  In other words, in $\PSLO$, 
$\Cohomol^*(\circlegraph) \cong \Cohomol^*(\Z \times \Z/2)$. By an argument similar to the one after Lemma \ref{circle-contribution}, the $d_1$ differential vanishes on $E_1^{0,q}$ for $q > 0$.  We conclude that the $E_2$ page of the ESS for $\Cohomol^*(\text{P}\Gamma)$ has the form
\[
\begin{array}{l | l r r}
3 & (\ef)^{k}  &   (\ef)^{k}        &   \\
2 & (\ef)^{k}  &   (\ef)^{k}        &   \\
1 & (\ef)^{k}  &   (\ef)^{k}        &   \\
0 & \ef          &  \oplus_{\beta^1} \ef  & \oplus_{\beta^2} \ef \\
\hline
   &  0          &  1    &   2
\end{array}
\]  
where $\beta^q = \dim_{\F_2}\Cohomol^q(_\Gamma \backslash X ; \thinspace \F_2)$.
Once the $\dE$ differential is determined in this spectral sequence, $E_3 = E_\infty$ and the calculation is complete. 
We note that in the ESS for $\Cohomol^*(\text{P}\Gamma)$, classes in $E_2^{1,q}$  with $q > 0$ survive to $E_\infty$ and are products of an exterior class with a polynomial class.  Classes in $E_2^{0,q}$ with $q > 0$ are polynomial by Property~\ref{ss3} in Section~\ref{ssec: specseq}.   

The set-up is similar for the calculation of $\Cohomol^*(\Gamma)$ using the ESS.  As noted above, the $2$-torsion subcomplex for $\Gamma$ is the same as the reduced non-central $2$-torsion subcomplex for $P\Gamma$.  However, although in $\SLO$ a $\circlegraph$ component is again a single edge and single vertex, this time the cell stabilizers are isomorphic to $\Z/4$.  As a graph of groups, this corresponds to an HNN extension $\Z/4 *_{\Z/4}$ which has the presentation $\langle t, x | x^ 4 = 1, t^{-1} x t = x^i \rangle$ with $i = \pm 1$.  Analogous to the argument presented after Lemma \ref{circle-contribution}, in this situation $d_1$ vanishes and $\Cohomol^*(\circlegraph) \cong \Cohomol^*(\Z \times \Z/4; \ef)$.  We note for later that $\dim_{\F_2} \Cohomol^j(\Z \times \Z/4; \ef) = 2$ when $j > 0$.

We next determine $\Cohomol^*(\Gamma)$  using the Lyndon-Hochschild-Serre spectral sequence associated to the extension 
\[
1 \rightarrow \Z/2 \rightarrow \Gamma \rightarrow \text{P}\Gamma \rightarrow 1.
\]
Analogous to the calculation of $\Cohomol^*(\Z/4)$ from $\Cohomol^*(\Z/2)$, it follows that $\dL_2 \neq 0$.  To analyze the image of this differential, let $z_1$ be the polynomial generator of the cohomology ring of the central $\Z/2$ which forms the vertical edge of the LH3S.  The horizontal edge of this spectral sequence can be identified with $\Cohomol^*(\text{P}\Gamma)$ with untwisted coefficients.  We will show that the class represented by $z_1^2$ survives to $E_\infty$; this will be our unique polynomial generator.

Assume that $\dL_2$ has a non-zero image in some $\circlegraph$ component of $\Cohomol^*(\text{P}\Gamma)$ on the horizontal edge of the spectral sequence.   By direct calculation, one can show that if in this component the image of $\dL_2$ is the exterior class from $\Cohomol^*(\circlegraph)$, then in the abutment, the $\F_2$-dimension of the resulting cohomology in dimension $k$ associated to that component is $k+1$.  This is not possible, as Theorem \ref{thm:E2page} implies that $\Cohomol^*(\Gamma)$  is $4$-periodic.   We conclude that if the image of $\dL_2$ is non-zero, then its image must be the polynomial class.  
By Property~\ref{ss2} in Section~\ref{ssec: specseq} there is a multiplicative structure in the ESS which is compatible with the cup product structure.  This structure implies that the image of $\dL_2$ cannot live in $E_2^{1,1}$ or $E_2^{2,0}$, so it must live in $E_2^{0,2}$.

Moving on to the $E_3$ page of the LH3S, we calculate using the compatibility between $Sq^1$ and the differential, applying the general version of Property~\ref{ss4} in Section~\ref{ssec: specseq}  as given in~\cite{Singer}*{Theorem 2.17}.  We have
\[
\dL_3 \left((z_1)^2\right) = \dL_3\left(Sq^1(z_1)\right) = Sq^1 \dL_2(z_1). 
\]
We claim that $Sq^1 \dL_2(z_1) = 0$.  The summands in $\dL_2(z_1)$ lie on the horizontal edge of the spectral sequence, hence they originate in $\Cohomol^2(\text{P}\Gamma)$.  From the analysis of the image of $\dL_2$ at the end of the prior paragraph, these classes originate in $E_2^{0,2}$, which in turn is generated by squared $1$-dimensional polynomial classes which map to $0$ under $Sq^1$.  Therefore, $\dL_3$ is the zero map on $E_2^{0,2}$.  We conclude that this class, represented by $z_1^2$, is a permanent cocycle and hence is the $2$-dimensional polynomial class in $\Cohomol^*(\Gamma)$.
\end{proof}

Combining Lemmas~\ref{lem:d24qVanish},~\ref{lem:BR33} and~\ref{lem:circle2periodic} with Theorem~\ref{thm:E2page} and Lemma~\ref{circle-contribution}, we obtain
\begin{corollary}
 Let the non-central $2$-torsion subcomplex $X_s$ admit as quotient $k \geq 1$ components of type $\circlegraph$.
Then $$\dim_{\ef} \Cohomol^q(\Gamma;\thinspace \ef) = 
\begin{cases}
 \beta^1 +k -r, & q = 1,\\
 \beta^2 +\beta^1 +k +c -r, & q \geq 2,\\
\end{cases}$$
 where  $r := \rank d_2^{0,1}$, $c$ the co-rank of Section~\ref{Section:E2 page}
 and $\beta^q = \dim_{\F_2}\Cohomol^q(_\Gamma \backslash X ; \thinspace \F_2)$.
\end{corollary}

This next result about arbitrary reduced $2$-torsion subcomplex components  is motivated by Lemmas 34 and 35 in~\cite{BerkoveRahm}.     

\begin{lemma} \label{lem:d2 mod3}
The $d_2$ differential is trivial in dimensions congruent to $ 3 \bmod 4$ 
on all non-central $2$-torsion subcomplex components which are not of type $\circlegraph$.
\end{lemma}

\begin{proof}
By Theorem~\ref{prop:restriction}, the only non-trivial restriction map on cohomology in odd dimensions for finite subgroups is $res^{\Q}_{\Z/4}(x_1)  = b_1$.  
In particular, the restriction maps on cohomology are zero in dimensions $4k + 3$, so $d_1$ is trivial on these classes.  
Hence, these classes survive to the $E_2$ page.  We note that $Sq^2$ is trivial on classes in both $\Cohomol^3(\Te)$ and $\Cohomol^3(\Q)$.  
The former follows since $\Cohomol^5(\Te) = 0$; the latter follows since $Sq^2(x_1^2 y_1) = x_1^4 y_1 = 0$ by ring relations in $\Cohomol^*(\Q)$.  
On the other hand, $d_2: E_2^{0,3} \rightarrow E_2^{2,2} \cong \bigoplus \Cohomol^2(\Z/2)$, so the image of $d_2$ lies in $\bigoplus_S \Cohomol^2(\Z/2)$ where the finite sum is over $2$-cells in its support.  
In the cohomology rings $\Cohomol^*(\Z/2)$ in this sum,  $Sq^2(z_1^2) = z_1^4$.  
Since $d_2 Sq^2 = Sq^2 d_2$, we conclude that $d_2$ must vanish on $E_2^{0,3}$, and more generally on $E_2^{0,4k+3}$ by periodicity.     
\end{proof}

Now we have all the ingredients for the proof of Theorem~\ref{mainthm} that was stated in the introduction.
\begin{proof}[Proof of Theorem~\ref{mainthm}] Theorem~\ref{thm:E2page} gives us the general form of $E_2$ page of the equivariant spectral sequence:

\[  
\begin{array}{l | cccl}
q \equiv 3 \mod 4 &  E_2^{0,3}(X_s)     &   E_2^{1,3}(X_s) \oplus (\ef)^{a_1}  &   (\ef)^{a_2} \\
q \equiv 2 \mod 4 &  \Cohomol^2_\Gamma(\xsp) \oplus (\ef)^{1 -\sign(v)}     &    (\ef)^{a_3}&   \Cohomol^2(_\Gamma \backslash X)  \\
q \equiv 1 \mod 4 &  E_2^{0,1}(X_s)     &  E_2^{1,1}(X_s) \oplus  (\ef)^{a_1} &   (\ef)^{a_2} \\
q \equiv 0 \mod 4   & \F_2  &   \Cohomol^1(_\Gamma \backslash X) &  \Cohomol^2(_\Gamma \backslash X) \\
\hline  &  p = 0 & p = 1 &  p = 2
\end{array}
\]
with 
\[\begin{array}{ll}
a_1 & =  \chi(_\Gamma \backslash X_s) -1 +\beta^1(_\Gamma \backslash X) +c   \\
a_2 & =  \beta^{2} (_\Gamma \backslash X) +c \\
a_3 & =  \beta^{1} (_\Gamma \backslash X) +v -\sign(v)
\end{array}
\]  

We determine components of the $E_2$ page of the spectral sequence as follows:

\begin{enumerate}
\item We have $\dim_{\ef}  \Cohomol^2_\Gamma(\xsp) = 2m$, since only $\Q$ has non-trivial cohomology in dimension $2$ and $\Cohomol^2_\Gamma(\Q) = \ef^2$.

\item To calculate $\chi(_\Gamma \backslash X_s)$, we note that the Euler characteristic takes the following values on the connected components of $_\Gamma \backslash X_s$:
\begin{itemize}
	\item $0$ on a component of type $\circlegraph$;
	\item $1$ on a component of type $\edgegraph$;
	\item $-1$ on a component of type $\thetagraph$ or $\dumbbellgraph$.
\end{itemize}
From the two vertices involved per component, we see that there are $\frac{n}{2}$ components of type $\edgegraph$
and $\frac{m}{2}$ components of type either $\thetagraph$ or $\dumbbellgraph$.  
Therefore, $\chi(_\Gamma \backslash X_s) = \frac{n}{2}-\frac{m}{2}$.  This implies that $a_1 =  \frac{n}{2}-\frac{m}{2} - 1 + \beta^1 + c$.

\item We use Lemma~\ref{circle-contribution} to determine contributions of the connected components of ${}_\Gamma \backslash X_s$ of type $\circlegraph$.
For the other three types of connected components, we apply the proof of Theorem~\ref{thm:generalcohom}, which states in odd dimensions:
\begin{table}[h]
\begin{center}
\begin{tabular}{c | c c c c }
Involved types & $\dim_{\F_2}E_2^{0,1}$ & $\dim_{\F_2}E_2^{0,3}$ & $\dim_{\F_2}E_2^{1,1}$ & $\dim_{\F_2}E_2^{1,3}$   \\
\hline
$\thetagraph$,  $\dumbbellgraph$, $\edgegraph$ & $m$ & $m+n$ & $\frac{m}{2}+\frac{n}{2}$ & $3\frac{m}{2}+\frac{n}{2}$  \\
\end{tabular}
\end{center}
\end{table}

\item  We have $v = m+n$.  Since either $m > $ or $n > 0$ , $\sign(v) = 1$.  We conclude that $a_3 = \beta^1 + m + n - 1$.
\end{enumerate}

We can now write down the $E_3 = E_\infty$ page of the spectral sequence.   We let $r^{0,q}$ be the $\ef$-rank of the $d_2$ differential $E_2^{0,q} \rightarrow E_2^{2,q-1}$.  We substitute in the calculated values from the list and simplify.  
For clarity, we write down the dimensions of the $\ef$-vector space. 
\small 
\[  
\begin{array}{l | llll}
q \equiv 4 \mod 4 & 1  &   \beta^1 &  \beta^2 -  r^{0,1} \\
q \equiv 3 \mod 4  &  m + n + k  - r^{0,3}  
&  m +n +k - 1 + \beta^1 + c &   \beta^2 + c \\
q \equiv 2 \mod 4  & 2m - r^{0,2}   &    \beta^1 + m + n - 1 &   \beta^2 -  r^{0,3}\\
q \equiv 1 \mod 4  &  m + k  - r^{0,1} 
&  n +k  -1 + \beta^1 + c  &   \beta^2 + c -  r^{0,2}\\
q = 0   & 1  &   \beta^1 &  \beta^2  - r^{0,1} \\
\hline &  p = 0 & p = 1 &  p = 2
\end{array}
\]
\normalsize

We conclude:

\begin{center}
$\dim_{\ef} \Cohomol^q(\Gamma;\thinspace \ef) = $\small$
\begin{cases}
 \beta^1 +m +k -r^{0,1}, & q = 1,\\
 \beta^2 +\beta^1 +2m+n +k -1 +c -r^{0,1} -r^{0,2}, & q \equiv 2 \mod 4,\\
  \beta^2 +\beta^1 +2(m+n) +k -1 +c -r^{0,3} -r^{0,2}, & q \equiv 3 \mod 4,\\
  \beta^2 +\beta^1 + m+n +k +c -r^{0,3}, & q \equiv 4 \mod 4, \\
    \beta^2 +\beta^1 +m +k +c -r^{0,1}, & q \equiv 5 \mod 4.\\
\end{cases}$\normalsize
\end{center}

To get the special cases, we apply our vanishing results for the $d_2$ differential:
\begin{itemize}
  \item Lemma~\ref{lem:d2 mod3} for the case $k = 0$,
  \item Lemma \ref{lem:d2evenvanish} for the case $c = 0$.
\end{itemize}
\end{proof}

We have not been able to prove a result about the vanishing/non-vanishing of the $d_2$ differential in dimensions $\equiv 1 \bmod 4$.  
However, there is an alternative way to derive this information using the Universal Coefficient Theorem,
\[
 \Cohomol^1(\Gamma; \ef) \cong \ {\rm Hom}(\Homol_1(\Gamma; \Z), \ef).
\]
The rank of the right hand side is readily determined by finding the abelianization of $\Gamma$. 
By the table in Theorem~\ref{thm:E2page}, this rank is also equal to 
\[
 {\dim}_{\ef}  (E_2^{0,1}(X_s))  + {\dim}_{\ef} \Cohomol^1({}_\Gamma \backslash X )   - \rank(d_2^{0,1}).
\]

\section{Ford Fundamental Domains}\label{sec:Ford}

In this section, we present an algorithm which constructs a fundamental domain for a congruence subgroup that allows us to efficiently extract torsion subcomplexes as well as to determine the number and type of connected component types---we use these data to evaluate the formulas of Theorem~\ref{mainthm}.
In particular, we construct a Ford fundamental domain for the congruence subgroup $\go(\pi) \subset \pslo$ where $\langle \pi \rangle$ is a prime ideal in $\calO_{-11} = \Z[\omega]$, the ring of integers of the number field $\rationals[\sqrt{-11}]$, and $\omega = (-1+\sqrt{-11})/2$. We prove that, given a prime ideal $\langle \pi \rangle \subset \calO_{-11}$, there exists a Ford domain for the congruence subgroup $\text{P}\Gamma_0(\pi) \subset \pslo$ which has a particular structure similar to that found by Orive~\cite{Orive} for congruence subgroups of $\mathrm{PSL}_2(\mathbb Z)$. 

Due to the specific structure of these congruence subgroups of prime level, the Ford domains have a uniform structure, with one face for each residue class modulo $p$, and a small fixed number of additional faces. Similar Ford domains may be found for congruence subgroups of this form in other Bianchi groups. We first recall the definition of such a fundamental domain.

Given $\gamma \in \mathrm{Isom}^+(\mathbb{H}^3) \cong \mathrm{PSL}_2(\bC)$, the \emph{isometric sphere} $S_\gamma$ of $\gamma$ is defined to be the (hemi)sphere on which $\gamma$ acts as a Euclidean isometry; it can be shown that if 
$ \gamma = \begin{pmatrix} a & b \\ c & d\end{pmatrix},$ 
then $S_\gamma$ has radius $1/|c|$ and center $-d/c \in \bC$. The isometric sphere $S_{\gamma^{-1}}$ of the inverse $\gamma^{-1}$ has the same radius and center $a/c$. A \emph{Ford fundamental domain} $F$ for $\text{P}\Gamma$ is then the intersection between the region $B \subset \mathbb{H}^3$, exterior to all isometric spheres, and a fundamental domain $F_\infty$ for the subgroup $\text{P}\Gamma_\infty \subset \text{P}\Gamma$ of elements which fix $\infty$. For background on Ford domains, see, for example, Maskit~\cite{Maskit}, Chapter II.H.

We shall call $\emph{visible spheres}$ those isometric spheres which contribute to the boundary of $B$, because they are the spheres one ``sees" when viewing the boundary of $B$ from above; beneath them are infinitely many smaller spheres, each completely covered by the collection of spheres which bound $B$. We say that an isometric sphere $S$ \emph{covers} a point $p \in \bC$ if $p$ lies in the interior of $S$, and we say that the sphere $S$ is covered by other spheres $\{ S_i \}$ if the union of the interiors of the $S_i$ contains the interior of $S$.

We note that a Ford domain for $\pslo$ is constructed as follows. The region $B$ exterior to all isometric spheres is bounded by those isometric spheres of radius 1 centered at points of the ring $\calO_{-11}$. The fundamental domain $F_\infty$ can be taken to have vertices at 
\begin{eqnarray} \label{z1}
 \pm z_1 = \pm(3/11 + 6\omega/11) =\pm 3i/\sqrt{11}, \\
 \pm z_2 = \pm(-3/11+5\omega/11) = \pm (-1/2+5i/(2\sqrt{11})),\\
\label{z3}  \pm z_3 = \pm(8/11+5\omega/11) = \pm (1/2+5i/(2\sqrt{11})),
\end{eqnarray}
 and thus the Ford domain is the convex hyperbolic polyhedron with vertices at the cusp $\infty$ 
 and above these six points of $\bC$ at height $\sqrt{2/11}$ (see Figure \ref{figure:full_ford}).
 
 \begin{figure}[htb]
 \begin{center}
 \includegraphics[scale=0.6]{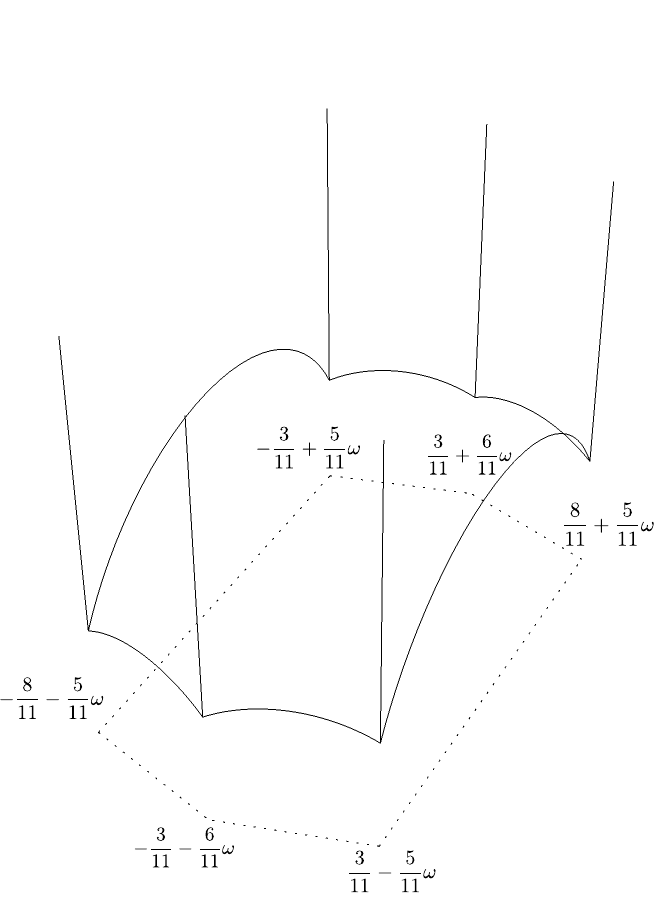}
 \caption{Ford domain for $\mathrm{PSL}_2(\mathcal{O}_{-11})$}
 \label{figure:full_ford}
 \end{center}
 \end{figure}

In the rest of this section, we will appeal to the following lemmas about isometric spheres of elements of $\pslo$.

\begin{lemma}\label{nozerocover}The only elements of $\pslo$ whose isometric spheres cover $0$ are those of the form
$ \begin{pmatrix}a & -1 \\ 1 & 0 \end{pmatrix} $
for $a \in \calO_{-11}$. \end{lemma}

\begin{proof}In order for an isometric sphere to cover $0$, the radius of the sphere must be strictly larger than the complex modulus of its center. For an element 
$ \begin{pmatrix}a & b \\ c & d\end{pmatrix} \in \pslo, $
the radius is $1/|c|$ and the complex modulus of the center is $| -d/c| = |d|/|c|$. Therefore, the isometric sphere covers $0$ if and only if we have $1 > |d|$ for  $d \in \calO_{-11}$, 
and hence we must have $d=0$. 
The requirement that the determinant be $1$ then means that we have $ad-bc = -bc=1$. 
The only non-zero elements of $\calO_{-11}$ of modulus at most $1$ are $\pm 1$, 
and so we see that $b = -c = \pm 1$. 
Since we are working in $\pslo$, we may choose $c=1$. \end{proof}

As a consequence of Lemma~\ref{nozerocover}, we see that there is only one isometric sphere covering $0$, and it is the sphere of radius $1$ centered at the cusp $0$.

\begin{lemma}\label{remove}If the isometric sphere of radius 1 centered at the cusp $0$ is removed from the Ford domain for $\pslo$, the isometric spheres which become visible are those corresponding to the elements
\[ \begin{pmatrix} 1 & 0 \\ \omega & 1 \end{pmatrix}, \begin{pmatrix} 1 & 0 \\ \omega +1 & 1 \end{pmatrix},\]
and their inverses. \end{lemma}

\begin{proof}To see this, consider the isometric spheres of radius 1 centered at elements of $\calO_{-11}$, and consider removing the sphere centered at the cusp $0$. We then see the isometric spheres of
\[ \begin{pmatrix} 1 & 0 \\ \omega & 1 \end{pmatrix}, \begin{pmatrix} 1 & 0 \\ \omega +1 & 1 \end{pmatrix},\]
and their inverses, visible between the spheres of radius 1 centered at $\pm 1$, $\pm \omega$ and $\pm \overline{\omega}$. These smaller spheres are centered at $\pm 1/6 \pm \sqrt{-11}/6$ and have radius $1/\sqrt{3}$. Together with the radius 1 spheres centered at $\pm 1$, these six spheres all intersect at the cusp $0$, and in vertical circles of radius $\sqrt{11}/6$. The smaller spheres intersect the radius 1 spheres centered at $\pm \omega$ and $\pm \overline{\omega}$ in the same locus which the removed sphere did.

By Lemma~\ref{nozerocover}, no other sphere can cover $0$. Another visible sphere which intersects $0$ must have radius at least $\sqrt{11}/6$, because it must be visible above the circles where the existing spheres intersect. But since this radius is larger than $1/2$, the only possibilities are that the radius could be $1$ or $1/\sqrt{3}$, and these matrices have already been considered. \end{proof}

In the following, we will denote by $\mathcal{B}$ the collection of isometric spheres of radius 1 centered at all elements of $\calO_{-11}$. 
Then, given an ideal $\langle \pi \rangle \subset \calO_{-11}$, we will denote by $\mathcal{B}_{\langle \pi \rangle}'$ 
the collection of isometric spheres $\mathcal{B}$ with those centered at elements of the ideal $\langle \pi \rangle$ removed. The spheres which then become visible are described by Lemma~\ref{remove}.

Let $p \in \Z$ be an odd rational prime.

\subsection{Case 1: $p$ splits in $\calO_{-11}$}

Suppose that $p = \pi \overline{\pi}$, where $\pi = a + b\omega$. We suppose that $a \geq 0$ and that $b>0$. We will find a Ford domain for the subgroup $\go(\pi)$. We first establish which isometric spheres form the boundary of the set $B$.

For each $\alpha \in \calO_{-11} \setminus \langle \pi \rangle$ there is a corresponding element
\[ \begin{pmatrix} * & * \\ \pi & -\alpha\end{pmatrix} \]
and for no $\alpha \in \langle \pi \rangle$ is there a similar element, as this would cause the determinant to be divisible by $\pi$. The isometric sphere of the given element has center $\alpha/\pi$ and radius $1/\sqrt{p}$. Consider the collection of spheres $\mathcal{S}_{\langle \pi \rangle}$ of radius $1/\sqrt{p}$ centered at $\alpha/\pi$ for $\alpha \in \calO_{-11} \setminus \langle \pi \rangle$.
 Include also in $\mathcal{S}_{\langle \pi \rangle}$ the isometric spheres of 
\[ \begin{pmatrix} 1 & 0 \\ \omega \pi & 1 \end{pmatrix}, \begin{pmatrix} 1 & 0 \\ (\omega +1) \pi & 1 \end{pmatrix},\]
their inverses, and the translates of these spheres by $\text{P}\Gamma_\infty$; these spheres have radius $1/\sqrt{3p}$. We claim that the spheres of $\mathcal{S}_{\langle \pi \rangle}$ suffice to determine the boundary of $B$.

Suppose for sake of contradiction that there exists an element 
\[ M=\begin{pmatrix} \alpha & \beta \\ \delta \pi & \gamma\end{pmatrix} \]
where $\delta \neq 1$, whose isometric sphere is visible above $\mathcal{S}_{\langle \pi \rangle}$. Then we may apply an isometry 
\[ \psi_\pi = \begin{pmatrix} \sqrt{\pi} & 0 \\ 0 & 1/\sqrt{\pi} \end{pmatrix} \]
of $\mathbb{H}^3$ which moves $\mathcal{S}_{\langle \pi \rangle}$ to $\mathcal{B}_{\langle \pi \rangle}'$ and conjugates $M$ to
\[ M'= \psi_\pi M \psi_\pi^{-1} = \begin{pmatrix} \alpha & \beta \pi \\ \delta & \gamma\end{pmatrix}. \]
By the assumption, the isometric sphere of $M'$ is visible above those of $\mathcal{B}_{\langle \pi \rangle}'$. But by Lemma~\ref{remove}, no such spheres are visible, and we have a contradiction. 

We now have a complete list of visible isometric spheres. We note that the vertices at which these spheres intersect are exactly those of $\mathcal{B}_{\langle \pi \rangle}'$ with the isometry $\psi_\pi^{-1}$ applied to them; the vertices of $\mathcal{B}_{\langle \pi \rangle}'$ are the vertices at $\pm z_1$, $\pm z_2$, $\pm z_3$, 
and their translates by $\text{P}\Gamma_\infty$, and at height $\sqrt{2/11}$, 
where $z_1$, $z_2$ and $z_3$ were defined in equations~(\ref{z1}) to~(\ref{z3}) above. After applying $\psi_\pi^{-1}$, 
these vertices are located at $(\pm z_1 + \alpha)/\pi$, $(\pm z_2 + \alpha)/\pi$, and $(\pm z_3 + \alpha)/\pi$, for $\alpha \in \calO_{-11}$, at height $\sqrt{2/11p}$. 

It remains to choose a fundamental domain for the action of $\text{P}\Gamma_\infty$. We claim that each isometric sphere of radius $1/\sqrt{p}$ is $\text{P}\Gamma_\infty$-equivalent to one centered at a point $c/\pi$ for $c \in \Z \setminus \langle p \rangle$, where $\langle p \rangle \subset \Z$ is the prime ideal generated by $p$. 
To see this, we wish to show that there exist $l, m \in \Z$ such that $\alpha/\pi + l + m\omega = c/\pi$ for $c \in \Z$. Writing $\alpha = x + y\omega$, we have 
\[ \frac{\alpha}{\pi} + l + m\omega = \frac{x+y\omega + (l+m\omega)\pi}{\pi}. \]
Writing $\pi = a + b\omega$,
\[ \frac{x+y\omega + (l+m\omega)\pi}{\pi} =  \frac{x+y\omega + al +am\omega+bl\omega + bm\omega^2}{\pi}. \]
Writing $\omega^2 = -3-\omega$,
\begin{align*} \frac{x+y\omega + al +am\omega+bl\omega + bm\omega^2}{\pi} 
 &= \frac{x + al -3bm + (y + am+bl - bm)\omega}{\pi}. \end{align*}
Since $a, b, l, m, x, y \in \Z$, the numerator is a rational integer if and only if 
\[ y +  am+bl - bm = y+bl+(a-b)m = 0. \]
But since $a$ and $b$ are necessarily relatively prime, it follows that so are $b$ and $a-b$, and so we may choose rational integers $l$ and $m$ so that $bl + (a-b)m = -y$.

We next note that none of the isometric spheres centered at $j/\pi$, for $1 \leq j \leq p-1$, are $\text{P}\Gamma_\infty$-equivalent, and that for all $j \in \Z$, $j/\pi$ is $\text{P}\Gamma_\infty$-equivalent to $(j+p)/\pi = (j/\pi) + \overline{\pi}$. Thus every visible sphere of $\mathcal{S}_{\langle \pi \rangle}$ is $\text{P}\Gamma_\infty$-equivalent to one of those centered at $j/\pi$, $1 \leq j \leq p-1$, or to one of those of radius $1/\sqrt{3p}$ which intersect at the cusp $0$. It is most convenient to have our fundamental domain have its faces on the isometric spheres centered at $j/\pi$ for $-(p-1)/2 \leq j \leq (p-1)/2$ (and $j \neq 0$).

We therefore take as our fundamental domain for $\text{P}\Gamma_\infty$ the region above these faces. 
This is the convex hull of the cusps $0$, $\infty$ and the vertices above $\pm z_1/\pi, \pm z_2/\pi$, and $\pm z_3/\pi$, and the following convex hulls: for each $j$ where $-(p-1)/2 \leq j \leq (p-1)/2$, the hull of the vertices at $(j \pm z_1)/\pi$, $(j+\pm z_2)/\pi$, $(j \pm z_3)/\pi$ at height $\sqrt{2/(11p)}$, and the vertex at $\infty$. See Figure \ref{figure:ford_5} for this domain when $p=5$, and $\pi = 2+\omega$.

 \begin{figure}[htb]
 \begin{center}
 \includegraphics[scale=0.6]{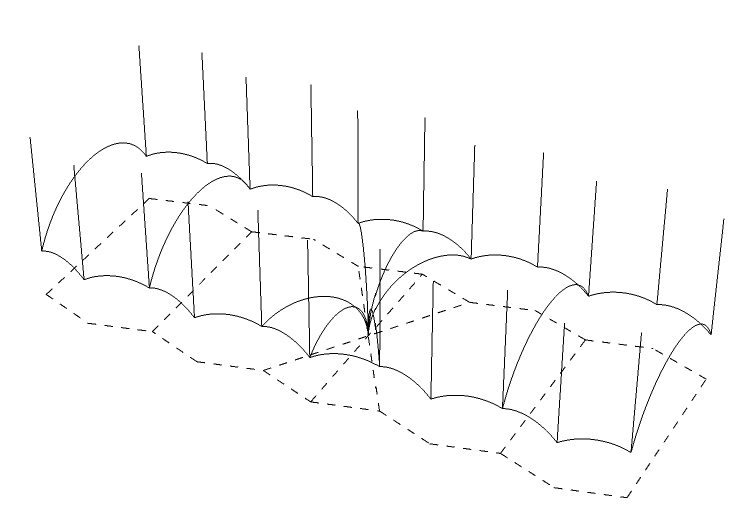}
 \caption{Ford domain for $\Gamma_0(2+\omega) \subset \mathrm{PSL}_2(\mathcal{O}_{-11})$, corresponding to $p=5$}
   \label{figure:ford_5}
 \end{center}
 \end{figure}

\subsection{Case 2: $p$ ramifies in $\calO_{-11}$}

In this case, we have $p=11$, and we will find a Ford domain for the group $\text{P}\Gamma_0(\pi)$ where $\pi = \sqrt{-11} = 1+2\omega$. This works the same as Case 1, with $a=1$ and $b=2$. Applying the method described above, we find a Ford domain bounded by isometric spheres of radius $1/\sqrt{11}$ and centered at $j/\sqrt{-11}$ for $j= \pm1, \pm 2, \pm 3, \pm 4,$ and $\pm 5$, along with four spheres of radius $1/\sqrt{33}$ centered at $\pm 1/(\omega\sqrt{-11})$ and $\pm 1/((\omega+1)\sqrt{-11})$. The vertices of this domain are at $0$ and $\infty$, and the finite vertices at $(j \pm z_1)/\sqrt{-11}$, $(j \pm z_2)/\sqrt{-11}$, and $(j \pm z_3)/\sqrt{-11}$ for $-5 \leq j \leq 5$ at height $\sqrt{2/121} = \sqrt{2}/11$ (see Figure \ref{figure:ford_11}).

 \begin{figure}[htb]
 \begin{center}
 \includegraphics[scale=0.6]{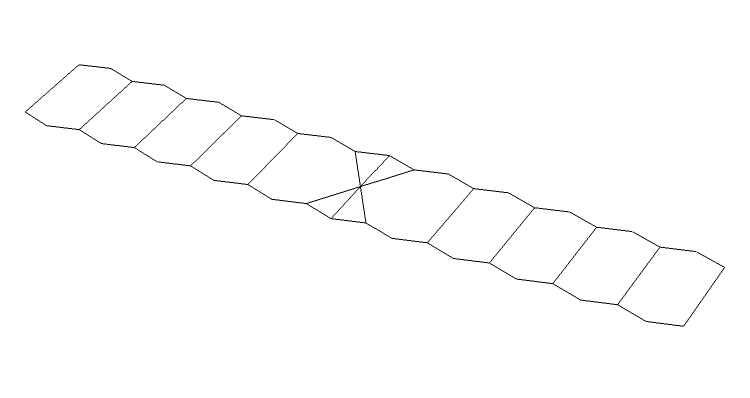}
 \caption{Projection to the complex plane of the Ford domain for $\Gamma_0(1+2\omega) \subset \mathrm{PSL}_2(\mathcal{O}_{-11})$, corresponding to $p=11$}
   \label{figure:ford_11}
 \end{center}
 \end{figure}

\subsection{Case 3: $p$ is inert in $\calO_{-11}$}

We will find a Ford domain for the group $\go(p)$. We first determine the isometric spheres which contribute to $B$.

For each element $\alpha \in \calO_{-11} \setminus \langle p \rangle$, there is a sphere of $B$ corresponding to the element
\[ \begin{pmatrix}\ast & \ast \\ p & -\alpha \end{pmatrix}. \]
We therefore have spheres of radius $1/p$ centered at all points $\alpha/p$ for $\alpha \in \calO_{-11}$, with those spheres centered at $(ap + bp\omega)/p = a+b\omega$ for $a, b \in \Z$ (i.e., those centered at points of $\calO_{-11}$) removed. Consider the collection of spheres $\mathcal{S}_{\langle p\rangle}$ which consists of these spheres and of the isometric spheres of the elements
\[ \begin{pmatrix} 1 & 0 \\ p\omega & 1\end{pmatrix}, \begin{pmatrix} 1 & 0 \\ p(\omega +1) & 1\end{pmatrix}, \]
their inverses, and the translates of these spheres by $\text{P}\Gamma_\infty$. We claim that the spheres of $\mathcal{S}_{\langle p\rangle}$ suffice to determine the boundary of $B$.

To see this, suppose that another isometric sphere, belonging to an element 
\[ M=\begin{pmatrix} \alpha & \beta \\ \delta p & \gamma\end{pmatrix}, \]
is visible above those of $\mathcal{S}_{\langle p\rangle}$. We conjugate by the element
\[ \psi_p = \begin{pmatrix} \sqrt{p} & 0 \\ 0 & 1/\sqrt{p} \end{pmatrix} \]
which has the effect of moving $\mathcal{S}_{\langle p\rangle}$ to $\mathcal{B}_{\langle p\rangle}'$. As such, the isometric sphere of the element
\[ M' = \psi_p M \psi_p^{-1} = \begin{pmatrix} \alpha & \beta p \\ \delta & \gamma\end{pmatrix} \]
is visible above those of $\mathcal{B}_{\langle p\rangle}'$. But by Lemma~\ref{remove}, no such spheres are visible, and we have a contradiction. 

We now have a complete list of visible isometric spheres. We note that the vertices at which these spheres intersect are exactly those of $\mathcal{B}_{\langle p\rangle}'$ with the isometry $\psi_p^{-1}$ applied to them; the vertices of $\mathcal{B}_{\langle p\rangle}'$ are the vertices at $\pm z_1$, $\pm z_2$, $\pm z_3$, 
and their translates by $\text{P}\Gamma_\infty$, and at height $\sqrt{2/11}$. After applying $\psi_p^{-1}$, 
these vertices are located at 
\begin{center}                              
  $(\pm z_1 + \alpha)/p$, $(\pm z_2 + \alpha)/p$, and $(\pm z_3 + \alpha)/p$,
\end{center}
for $\alpha \in \calO_{-11}$, at height $\sqrt{2/(11p^2)}$. 

It remains to choose a fundamental domain for the action of $\text{P}\Gamma_\infty$. Each of the visible spheres of radius $1/(p\sqrt{3})$ is $\text{P}\Gamma_\infty$-equivalent to one of those which intersects $0$. Each sphere of radius $1/p$ is $\text{P}\Gamma_\infty$-equivalent to one of those centered at $j/p + k\omega/p$ for $j, k$ integers with $-(p-1)/2 \leq j, k \leq (p-1)/2$ and $j, k$ not both $0$. Since these spheres are not pairwise $\text{P}\Gamma_\infty$-equivalent, we may take the vertices where these spheres intersect (including vertices where they intersect spheres adjacent to these spheres) as the vertices of our fundamental domain. The vertices are therefore located at $(z_1 + j + k\omega)/p$, $(-z_1 + j + k\omega)/p$, $(z_2 + j + k\omega)/p$, $(-z_2 + j + k\omega)/p$, $(z_3 + j + k\omega)/p$, and $(-z_3 + j + k\omega)/p$, for  $-(p-1)/2 \leq j, k \leq (p-1)/2$. These vertices are located at height $\sqrt{2/(11p^2)}$. 
Our fundamental domain is the union of the convex hull of the vertices above $\pm z_1/p, \pm z_2/p, \pm z_3/p$ and the cusps $0$ and~$\infty$, and each copy of the domain for $\mathrm{PSL}_2(\mathcal{O}_{-11})$ with vertices at  $(z_1 + j + k\omega)/p$, $(-z_1 + j + k\omega)/p$, $(z_2 + j + k\omega)/p$, $(-z_2 + j + k\omega)/p$, $(z_3 + j + k\omega)/p$, and $(-z_3 + j + k\omega)/p$ and $\infty$. See Figure \ref{figure:ford_7} for the case $p=7$.

 \begin{figure}[htb]
 \begin{center}
 \includegraphics[scale=0.6]{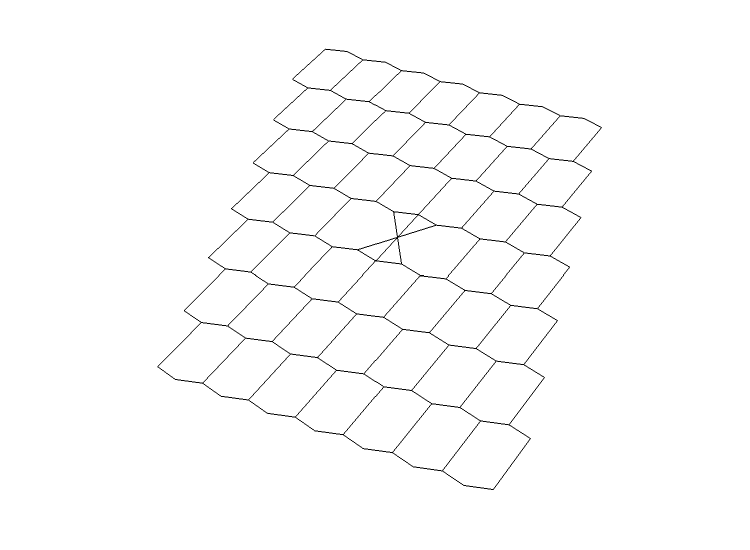}
 \caption{Projection to the complex plane of a Ford domain for $\Gamma_0(7) \subset \mathrm{PSL}_2(\mathcal{O}_{-11})$}
   \label{figure:ford_7}
 \end{center}
 \end{figure}


\subsection*{Other Bianchi Groups}

A similar method produces Ford fundamental domains for congruence subgroups 
$\go(\pi)$ of prime level in the Bianchi groups $\pslod$, 
where $\calO_{-m}$ is another Euclidean ring,
i.e. the ring of integers in the imaginary quadratic number field 
$\rationals(\sqrt{-m}\thinspace)$, 
for $m=1,2,3,7$. The argument above may be adapted to show that the set 
$B$ is bounded by isometric spheres of radius $1/|\pi|$ centered at $a/\pi$ 
for $a \in \calO_{-m} \setminus \langle \pi \rangle$, 
and by at most six smaller isometric spheres, 
which meet at the cusp $0$, and their translates by $\text{P}\Gamma_\infty$. 
Note that the set of spheres in question is invariant under the action of $\text{P}\Gamma_\infty$, hence the local picture is the same at every point of $\calO_{-m}$: 
at every such point, one sees translated copies of those spheres one sees at $0$. 
Furthermore, the fundamental domain for the action of $\text{P}\Gamma_\infty$ may be chosen so that the resulting fundamental domain is bounded by $N(\pi)-1$ spheres of radius $1/|\pi|$ 
(where $N(\pi) = |\pi|^2 \in \N$ denotes the norm of the prime element $\pi \in \calO_{-m}$), 
by the smaller spheres which meet at the cusp $0$, and by the vertical sides which make up the boundary of the fundamental domain of the action of $\text{P}\Gamma_\infty$. 

The fundamental domains constructed in this section for $\go(\pi)$ are composed of $N(\pi)+1$ copies of a Ford domain for the Bianchi group $\pslod$, 
where $N(\pi)+1 = [\pslod : \go(\pi)]$. 
We believe that it should be possible to construct fundamental domains for congruence subgroups of prime level in the other Bianchi groups in a similar way. Furthermore, it should be possible to construct Ford domains for composite levels from copies of those for prime levels in a fashion akin to that described by Lascurain Orive~\cite{Orive} in the case of $\mathrm{PSL}_2(\Z)$.

\section{Example computations} \label{Example computations}
We print the details for only two of our example calculations.
Checking the other outcomes of example calculations mentioned in Section~\ref{sec:components} 
is made straightforward by the algorithm of Section~\ref{sec:Ford}.
\begin{figure}
 \begin{subfigure}[b]{0.45\textwidth}
   \includegraphics[height=50mm]{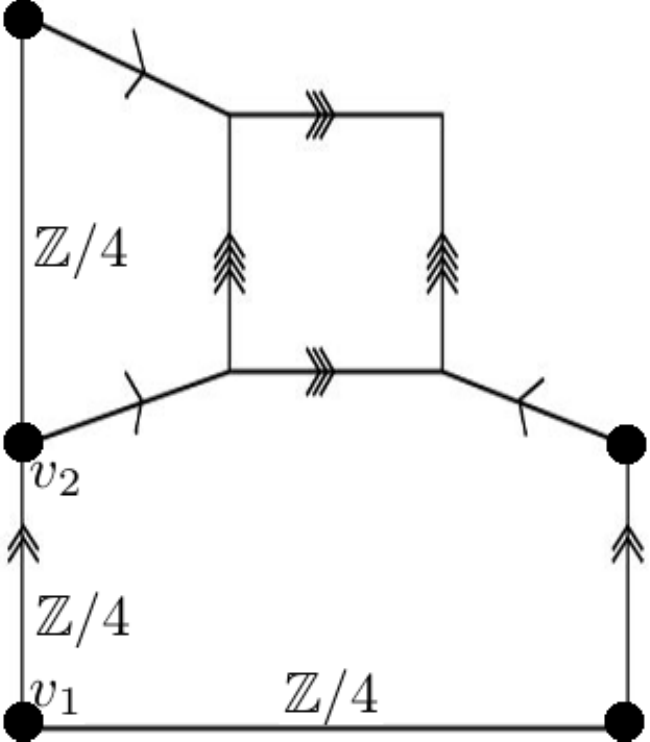}
  \caption{Strict fundamental domain for the $2$-dimensional equivariant retract (away from the principal ideal cusps).  The domain is 
      extracted from a Ford fundamental domain for $\Gamma_0(\sqrt{-2})$,
after the Borel--Serre bordification, replacing the cusps by Euclidean planes with wallpaper group action
(hence a rectangle with identified sides in the fundamental domain).}
  \label{fundamental domain}
 \end{subfigure}
 \hfill
 \begin{subfigure}[b]{0.45\textwidth}
   \includegraphics[height=40mm]{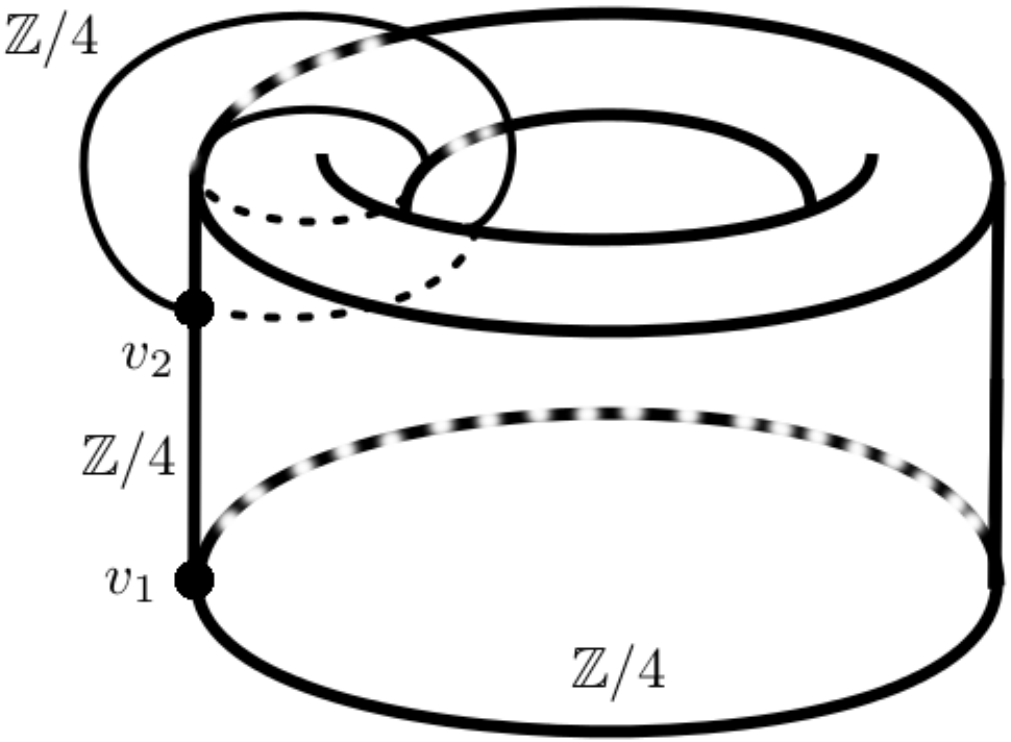}
  \caption{Orbit space of the displayed fundamental domain for $\Gamma_0(\sqrt{-2})$,
  with the cusp having become a $2$-torus in the Borel--Serre compactification.
  The $2$--torsion subcomplex is given by the edges labeled by $\Z/4$,
  and their vertices.}
  \label{orbit space}
  \end{subfigure} 
  \caption{\mbox{Fundamental domain for $\Gamma_0(\sqrt{-2}) \subset \SLtwo(\Z[\sqrt{-2}])$}}
\end{figure}
\subsection{Level $\sqrt{-2}$ in SL$_2(\Z[\sqrt{-2}])$}
Let 
$$\Gamma := {\Gamma_0(\sqrt{-2})}  := \left\{ \left. \tiny \mat \normalsize \in {\rm SL}_2(\Z[\sqrt{-2}]) \medspace \right| \medspace c \in \langle \sqrt{-2} \rangle \right\}.$$
A fundamental domain for $\Gamma$ is given in Figure~\ref{fundamental domain}.
Then, the orbit space including the non-central $2$--torsion subcomplex is drawn in Figure~\ref{orbit space}.
Hence $_\Gamma \backslash X_s = \dumbbellgraph$, $v = 2$, $\chi(_\Gamma \backslash X_s) = -1$. 
Using Figure~\ref{orbit space}, we convince ourselves that $c = 0$,
$\beta^1(_\Gamma \backslash X) = 2$ and $\beta^2(_\Gamma \backslash X) = 1$.
From Theorem~\ref{thm:generalcohom}, we see that the non-central $
2$--torsion subcomplex contributes the following dimensions to the $E_2$ page.
$$
\begin{array}{l | clcl}
q\equiv 3 \mod 4  & (\ef)^2           &  &  (\ef)^3    \\
q\equiv 2 \mod 4  & (\ef)^4           &  &  (\ef)^3    \\
q\equiv 1 \mod 4  & (\ef)^2           &  &  \ef           \\
q\equiv 0 \mod 4  & \ef                  &  &  (\ef)^2    \\
\hline & p = 0 & & p = 1
\end{array}
$$

From Lemma~\ref{lem:d2 mod3} and applying the method described just below it, 
we know that the $d_2$ differential of the equivariant spectral sequence vanishes.
Applying Theorem~\ref{thm:E2page}, this allows us to conclude
that the dimensions of the cohomology ring that we are looking for are
\begin{center}$
\dim_{\F_2}\Cohomol^{q}(\Gamma_0(\sqrt{-2}); \thinspace \F_2)
=$\scriptsize$
\begin{cases}
   c+ \beta^{1} (_\Gamma \backslash X) +\beta^{2} (_\Gamma \backslash X)+2, &  q = 4k+5, \\
   c +\beta^1(_\Gamma \backslash X) +\beta^{2} (_\Gamma \backslash X)+2, &  q = 4k+4, \\
   c +\beta^{1} (_\Gamma \backslash X) + \beta^{2} (_\Gamma \backslash X)+3, &  q = 4k+3, \\
   c+\beta^1(_\Gamma \backslash X) +\beta^{2} (_\Gamma \backslash X)+3, &  q = 4k+2, \\
   \beta^{1} (_\Gamma \backslash X)+2 ,  &  q = 1 \\
\end{cases}
=
\begin{cases}
   5, &  q = 4k+5, \\
   5, &  q = 4k+4, \\
   6, &  q = 4k+3, \\
   6, &  q = 4k+2, \\
   4 ,  &  q = 1. \\
\end{cases}
$\normalsize\end{center}

\subsection{Level $5$ in SL$_2(\Z[\sqrt{-2}])$} Let
$$\Gamma := {\Gamma_0(5)}  := \left\{ \left. \tiny \mat \normalsize \in {\rm SL}_2(\Z[\sqrt{-2}]) \medspace \right| \medspace c \in \langle 5 \rangle \right\}.$$
With the algorithm described in Section~\ref{sec:Ford}, 
we produce the Ford fundamental domain displayed in Figure~\ref{Ford domain}.
\begin{figure}
 \begin{subfigure}[b]{0.45\textwidth}
   \includegraphics[height=71mm]{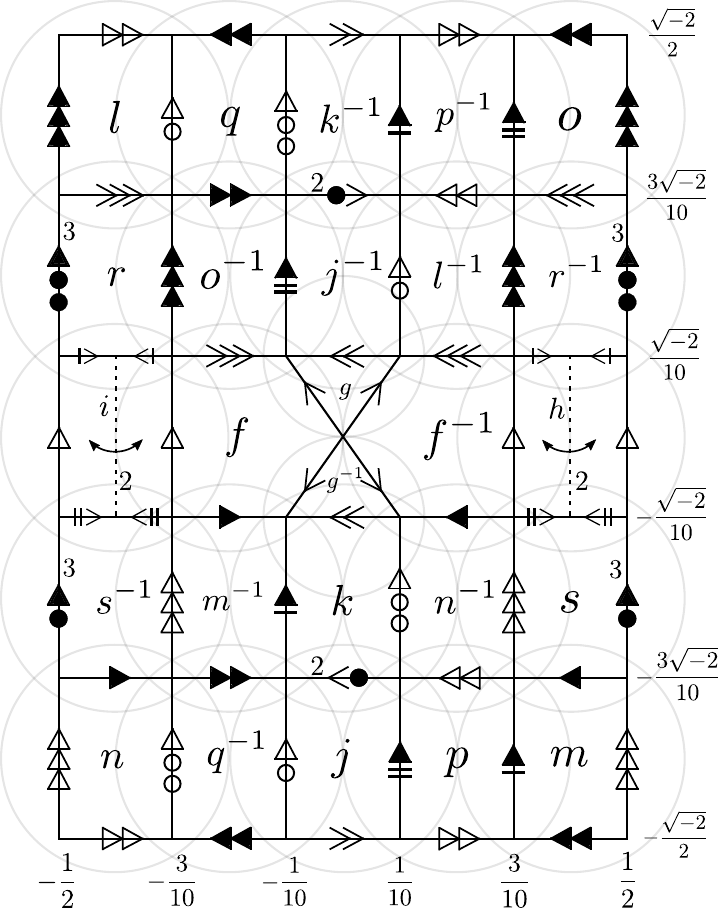}
 \caption{Bottom facets of a Ford fundamental domain for $\Gamma_0(5)$.}
  \label{Ford domain}
 \end{subfigure}
 \hfill
 \begin{subfigure}[b]{0.45\textwidth}
   \includegraphics[height=60mm]{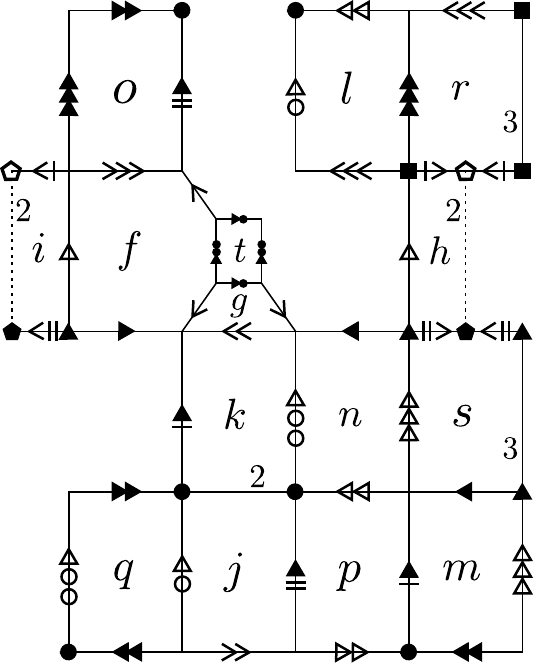}
      \caption{Strict fundamental domain for the 2-dimensional equivariant retract (away from the principal ideal cusps).  The domain is 
      extracted from the displayed Ford fundamental domain for $\Gamma_0(5)$,
    replacing the cusp by a rectangle in the Borel--Serre bordification 
    (passing to a $2$-torus in the Borel--Serre compactification).
  The $2$--torsion subcomplex is given by the edges labeled by a digit~$2$, and their vertices.}
    \label{strict_fundamental_domain}
 \end{subfigure}
 \caption{Ford fundamental domain for $\Gamma_0(5)$ in SL$_2(\Z[\sqrt{-2}])$}
\end{figure}
Carrying out the side identifications, under which each face of the Ford domain has a conjugate face,
we are left with the strict (in its interior) fundamental domain of 
Figure~\ref{strict_fundamental_domain},
which is subject to the indicated edge identifications.
Then the $2$-cells boundary matrix $\partial_2$ has elementary divisors 
\begin{itemize}
 \item 1, of multiplicity 12
\item 2 and 4, each of multiplicity 1.
\end{itemize}
Its kernel is one-dimensional.
The edges boundary matrix 
$\partial_1$ has the only elementary divisor 1, of multiplicity 6.
Therefore, the cellular chain complex of the orbit space,
$$\Z^{15} \overset{\partial_2}{\longrightarrow}  \Z^{22} \overset{\partial_1}{\longrightarrow} \Z^7,$$
has homology 
\[
\Homol_2(_\Gamma \backslash X, \thinspace \Z) \cong \Z, \ \Homol_1(_\Gamma \backslash X, \thinspace \Z) \cong \Z^2 \oplus\Z/2\Z \oplus\Z/4\Z, \  \Homol_0(_\Gamma \backslash X, \thinspace \Z) \cong \Z.
\]

Hence the remaining parameters for Theorem~\ref{thm:E2page} are $\beta^1 = 4$,
$\beta^2 = 3$, $a_3 = 4$, $a_2 = 3+c$ and $a_1 = 3+c$.

From Figure~\ref{strict_fundamental_domain}, we see that
$_\Gamma \backslash X_s \cong \circlegraph \circlegraph$, so
$\chi(_\Gamma \backslash X_s) = 0 = v$;
furthermore for the quotient of the $3$-torsion subcomplex, we also read off the type $\circlegraph \circlegraph$.
Let $\text{P}\Gamma$ be the quotient group $\Gamma / (\{1, -1\} \cap  \Gamma )$.
Because of non-adjacency of the torus $2$-cell to the torsion subcomplexes,
the $d^2_{2,0}$-differential of the equivariant spectral sequence converging to $\Homol_{p+q}(\text{P}\Gamma, \thinspace \Z)$
vanishes; and therefore this spectral sequence yields the short exact sequence
\[
 1 \to (\Z/2)^2 \oplus  (\Z/3)^2 \to \Homol_{1}(\text{P}\Gamma, \thinspace \Z) \to \Z^2 \oplus \Z/2 \oplus \Z/4 \to 1,
\]
which is compatible with the abelianization of the presentation for $\text{P}\Gamma$
produced with the algorithm of Section~\ref{sec:Ford},
as well as with Sch\"onnenbeck's machine result
\[
\Homol_{1}(\text{P}\Gamma, \thinspace \Z) \cong \Z^2 \oplus (\Z/12)^2 \oplus \Z/2.
\]
We now apply Theorem~\ref{thm:E2page}, and get dim $\Homol^1(\Gamma, \thinspace \ef) = 6 - \text{rank}d_2^{0,1}$.
Using the abelianization that we get from the fundamental domain in Figure~\ref{Ford domain},
or alternatively Sch\"onnenbeck's machine computation described in the Appendix,
we infer $\text{rank}\ d_2^{0,1} = 1$.

\section*{Acknowledgment}  We would like to thank an anonymous referee, whose useful comments improved the quality of this paper. 

\newpage

\begin{appendix}
 \section{Machine computations}

\begin{table}  \tiny
$\begin{array}{|c|c|c|c|c|c|c|c|c|c|c|c|c|c|c|} \hline
\Delta  &  -m  &  \text{Level } \eta  &  _{\Gamma_0(\eta)} \backslash X_s  &  \beta_1  & \beta^1  &  r  &  c  &  \Cohomol^1 \rule[-2ex]{0pt}{5ex}  &  \Cohomol^2  &  \Cohomol^3  &  \Cohomol^4  &  \Cohomol^5  \\
\hline &&&&&&&&&&&&\\
 -3  &  -3  &  \langle 2\rangle \text{ (note }\chi=2)&   \edgegraph  &  0  &  0  &  0  &  0  &  0  &  2  &  4  &  3  &  1   \\ 
 -11  &  -11  &  \langle 2\rangle  &  \edgegraph  &  2  &  2  &  0  &  0  &  2  &  4  &  6  &  5  &  3  \\  
 -15  &  -15  &  \langle 2, \frac{3+\sqrt{-m}}{2}\rangle  &  \circlegraph  &  4  &  4  &  0  &  0  &  5  &  8  &  8  &  8  &  8  \\ 
 -7  &  -7  &  \langle \frac{1-\sqrt{-m}}{2}\rangle 
 &  \circlegraph  &  2   &  2  &  0  &  0  &  3  &  4  &  4  &  4  &  4  \\ 
 -15  &  -15  &  \langle \frac{5-\sqrt{-m}}{2}\rangle  &  \circlegraph \circlegraph  &  8   &  8+r  &  r \leq 1  &  1 -r  &  10  &  18  &  18  &  18  &  18  \\ 
 -7  &  -7  &  \langle 3\frac{1-\sqrt{-m}}{2}\rangle 
 &  \circlegraph \circlegraph  &  4  &  4+r  &  r \leq 1  &  1 -r  &  6  &  10  &  10  &  10  &  10  \\ 
  -8  &   -2  &  \langle 5 \rangle &  \circlegraph \circlegraph  & 2 & 3+r & r \leq 2 & 2-r & 5 & 9 & 9 & 9 & 9\\
  -8  &   -2  &  \langle 3+2\sqrt{-m} \rangle &  \circlegraph \circlegraph  & 2 & 2+r & r \leq 2 & 2-r & 4 & 7 & 7 & 7 & 7\\
  -56  &  -14  &  \langle 2, \sqrt{-m}\thinspace \rangle  &  \circlegraph \circlegraph \circlegraph  &  10  &  10  &  1  &  0  &  12  &  21  &  21  &  21  &  21  \\ 
 -56  &  -14  &  \langle 2\rangle  &  \circlegraph \circlegraph \circlegraph  &  16  &  16  &  0  &  0  &  19  &  34  &  34  &  34  &  34  \\ 
 -52  &  -13  &  3\langle 2, 1+\sqrt{-m}\thinspace \rangle  &  \circlegraph \circlegraph \circlegraph  &  25  &  30+r  &  r \leq 3  &  3 -r  &  33  &  65  &  65  &  65  &  65  \\ 
 -40  &  -10  &  \langle 2+\sqrt{-m}\thinspace \rangle  &  \circlegraph \circlegraph \circlegraph  &  10  &  10+r  &  r \leq 2   &  2-r  &  13  &  24  &  24  &  24  &  24  \\ 
 -40  &  -10  &  \langle 5\rangle  &  \circlegraph \circlegraph \circlegraph  &  25  &  26+r  &  r \leq 3  &  3 -r  &  29  &  57  &  57  &  57  &  57  \\ 
 -40  &  -10  &  \langle 5, \sqrt{-m}\thinspace \rangle  &  \circlegraph \circlegraph \circlegraph  &  5  &  6+r  &  r \leq 3  &  3 -r  &  9  &  17  &  17  &  17  &  17  \\ 
 -24  &  -6  &  \langle 2\rangle  &  \circlegraph \circlegraph \circlegraph  &  7  &   7  &  0  &  0  &  10  &  16  &  16  &  16  &  16  \\ 
 -20  &  -5  &  \langle 10, 5+\sqrt{-m}\thinspace \rangle  &  \circlegraph \circlegraph \circlegraph  &  8  &  8+r  &  r \leq 3  &  3 -r  &  11  &  21  &  21  &  21  &  21  \\ 
 -40  &  -10  &  \langle \sqrt{-m}\thinspace \rangle  & 5\circlegraph  &  10   &  12+r  &  r \leq 5  &  5 -r  &  17  &  33  &  33  &  33  &  33  \\ 
 -24  &  -6  &  \langle 2+\sqrt{-m}\thinspace \rangle  &  5\circlegraph  &  8   &  8+r  &  r \leq 3  &  3 -r  &  13  &  23  &  23  &  23  &  23  \\ 
 -56  &  -14  &  \langle 10, 4+\sqrt{-m}\thinspace \rangle  &  6\circlegraph  &  18  &  18+r  &  r \leq 5  &  5 -r  &  24  &  45  &  45  &  45  &  45  \\ 
 -40  &  -10  &  \langle 2, \sqrt{-m}\thinspace \rangle  &  \circlegraph \thetagraph  &  5   &  5+r  &  r \leq 1  &  1 -r  &  8  &  14  &  14  &  13  &  13  \\ 
 -24  &  -6  &  \langle 2, \sqrt{-m}\thinspace \rangle  &  \circlegraph \thetagraph  &  4   &  4  &  0  &  1  &  7  &  11  &  11  &  11  &  11  \\ 
 -52  &  -13  &  \langle 2, 1+\sqrt{-m}\thinspace \rangle  &  \thetagraph  &  7  &  7  &  0  &  0  &  9  &  16  &  16  &  15  &  15  \\ 
 -20  &  -5  &  \langle 2, 1+\sqrt{-m}\thinspace \rangle  &  \thetagraph  &  4   &  4  &  0  &  0  &  6  &  10  &  10  &  9  &  9  \\ 
 -8  &  -2  &  \langle 2\rangle  &  \thetagraph  &  3  &  3  &  0  &  0  &  5  &  8  &  8  &  7  &  7  \\ 
 -52  &  -13  &  \langle 2 \rangle  &  \thetagraph \thetagraph  &  12  &  12  &  0  &  0  &  16  &  30  &  30  &  27  &  27  \\ 
 -40  &  -10  & \langle 2 \rangle  &  \thetagraph \thetagraph  &  9  &  9  &  0  &  0  &  13  &  24  &  24  &  21  &  21  \\ 
 -20  &  -5  &  \langle 2 \rangle  &  \thetagraph \thetagraph  &  6  &  6  &  0  &  0  &  10  &  18  &  18  &  15  &  15  \\ 
 -8  &  -2  &  \langle \sqrt{-m}\thinspace \rangle  &  \dumbbellgraph  &  2  &  2  &  0  &  0  &  4  &  6  &  6  &  5  &  5  \\ 
\hline
\end{array}$
\normalsize
\caption{Cases $\Gamma_0(\eta)$ in the sample with non-empty non-central $2$-torsion subcomplex}
\label{table:FullResults}
\end{table}

\begin{table}[h] \tiny
	$\begin{array}{|c|c|c|c|c|c|} 
	\hline &&&&&\\
	\Delta  & 	-56 & 	-56 & 	-56 & 	-52 & 	-52  \\
	-m 	& 	-14 & 	-14 & 	-14 & 	-13 & 	-13  \\
	\eta & 	\langle 3, 1+\sqrt{-m}\rangle  & 	\langle 6, 2+\sqrt{-m}\rangle  & 	\langle 9, 2+\sqrt{-m}\rangle  & 	\langle 7, 1+\sqrt{-m}\rangle  & 	\langle 1-\sqrt{-m}\rangle  \\
	\hline \beta^1   & 	10 & 	21 & 	20 & 	6 & 	14 \\
	\hline 
	\end{array}$
	$\begin{array}{|c|c|c|c|c|c|c|} 
	\hline &&&&&&\\
	\Delta & 	-52& -24&	-24 & 	-24 & 	-20 & 	-20  \\
	-m 	& 	-13&-6&	-6 & 	-6 & 	-5 & 	-5  \\
	\eta & 	\langle 6+\sqrt{-m}\rangle&\langle 3, \sqrt{-m}\rangle&  \langle \sqrt{-m}\rangle  & 	\langle 3\rangle  & 	\langle 3, 1+\sqrt{-m}\rangle  & 	\langle 1+\sqrt{-m}\rangle  \\
	\hline \beta^1 & 	20 &5   & 10 & 11 & 	4 & 	8 \\
	\hline
	\end{array}$
	$\begin{array}{|c|c|c|c|c|c|c|} 
	\hline &&&&&&\\
	\Delta & 	-20 & 	-15 & -15 & 	-15 & 	-15 & 	-11 \\
	-m 	& 	-5 & 	-15 &-15 & 	-15 & 	-15 & 		-11 \\
	\eta & 	\langle 2-\sqrt{-m} \rangle  & 	\langle 3, \sqrt{-m}\rangle & 	\langle \frac{1-\sqrt{-m}}{2}\rangle  & 	\langle \frac{3+\sqrt{-m}}{2} \rangle &\langle 3\rangle & 	 \langle \frac{1-\sqrt{-m}}{2}\rangle \\
	\hline \beta^1& 	8 & 	4 &	6 & 	8 & 	8 & 2 \\
	\hline
	\end{array}$
	$\begin{array}{|c|c|c|c|c|c|c|c|} 
	\hline &&&&&&\\
	\Delta & 	-11 & 	-11 & 	-11 & 	-11 & 	-8 & 	-8 \\
	-m & 	-11 & 	-11 & 	-11 & 	-11 & 	-2 & 	-2 \\
	\eta   & 	\langle 4\rangle  & 	\langle 1-\sqrt{-m}\rangle  & 	\langle \frac{5+\sqrt{-m}}{2}\rangle  & 	\langle 2+\sqrt{-m}\rangle  & 	\langle 1-\sqrt{-m}\rangle  & 	\langle 2+\sqrt{-m}\rangle  \\
	\hline \beta^1  & 6 & 4 & 	4 & 	5 & 	2 & 	4 \\
	\hline
	\end{array}$
	$\begin{array}{|c|c|c|c|c|c|c|c|c|c|} 
	\hline &&&&&&&&&\\
	\Delta & 		-7 & 	-7 & 	-7 & 	-7 & 	-3 & 	-3 & 	-3 & 	-3 & 	-3 \\
	-m & 			-7 & 	-7 & 	-7 & 	-7 & 	-3 & 	-3 & 	-3 & 	-3 & 	-3 \\
	\eta & \langle \frac{3+\sqrt{-m}}{2}\rangle   & 	\langle \sqrt{-m}\rangle  &\langle \frac{7+\sqrt{-m}}{2}\rangle  & 	\langle 7\rangle   & 	\langle \sqrt{-m}\rangle  & 	\langle 4\rangle  & 	\langle 2\sqrt{-m}\rangle  & 	\langle 3\rangle  & 	\langle \frac{9-\sqrt{-m}}{2}\rangle \\
	\hline \beta^1    &	3  & 	2  & 	4  & 	8  & 	0 & 	1 & 	0 & 	0 & 	1\\
	\chi &	   &	   &	   &	   &     2 &	2  &	4  &	4  &	4  \\
	\hline
	\end{array}$
	\normalsize
	\caption{Cases $\Gamma_0(\eta)$ in the sample with empty non-central $2$-torsion subcomplex} 
	\label{table:EmptyCmplx}
\end{table}

 We use Sebastian Sch\"onnenbeck's implementation~\cite{Schoennenbeck} of the Vorono\"i cell complex
 to compute the cohomology of a sample (six levels $\eta$ in each of ten imaginary quadratic fields)
 of congruence subgroups $\Gamma_0(\eta)$.
 Then we use Bui Anh Tuan's implementation of Rigid Facets Subdivision in order to extract the 
 non-central $2$-torsion subcomplex.
 This allows us to check our example computations with the algorithm of Section~\ref{sec:Ford},
 and to illustrate which values the parameters in our formulas can take.

Let $\ringOm$ be the ring of integers in the field $\rationals(\sqrt{-m})$ with discriminant $\Delta$. 
We present the ideal $\eta \subset \ringOm$ with the smallest possible number of generators;
hence when we use two generators, it is because $\eta$ is not principal.
We let $c$ be the co-rank defined in Section~\ref{Section:E2 page};
let $\beta_q = \dim_{\rationals}\Cohomol_q(_{\Gamma_0(\eta)} \backslash \hy ; \thinspace \rationals)$ and 
$\beta^q = \dim_{\F_2}\Cohomol^q(_{\Gamma_0(\eta)} \backslash \hy ; \thinspace \F_2)$ for $q = 1, 2$; 
and $_{\Gamma_0(\eta)} \backslash X_s$  
the orbit space of the non-central $2$-torsion subcomplex.
We further write $\Cohomol^q := \dim_{\F_2}\Cohomol^q({\Gamma_0(\eta)} ; \thinspace \F_2)$.
Let $r$ be the rank of the $d_2^{0,1}$-differential of the equivariant spectral sequence.
In many cases, notably for $c = 0$, or when there are no components different from $\circlegraph$ in $_{\Gamma_0(\eta)} \backslash X_s$,
the $d_2^{0,2+4k}$-differential vanishes because of the lemmata in Section~\ref{sec: d2 vanishes}.
\textit{Note that for one case in our sample, $\Gamma_0(\langle 2, \sqrt{-6}\rangle)$ in SL$_2(\Z[\sqrt{-6}\thinspace])$,
the machine computation yields} rank$(d_2^{0, 2+4k}) = 1$. 
In all other cases in our sample, the computation yields $d_2^{0, 2+4k} = 0$, so we do not print $d_2^{0, 2+4k}$.
The machine calculations in HAP~\cite{HAP} allow us to produce $\Cohomol^q$ and $\beta_1$ directly from Sebastian Sch\"onnenbeck's cell complexes.
From $\Cohomol^q$ and $\beta_1$, we use the corollaries in Section~\ref{sec: d2 vanishes} to get  $\beta^1$,  $r$ and $c$ in Table \ref{table:FullResults}.  
When $_{\Gamma_0(\eta)} \backslash X_s$ is empty, then because of Proposition~\ref{empty case},
only the Betti numbers are of interest; results for those cases can be found in Table \ref{table:EmptyCmplx}.  
In all cases except for the Eisenstein integers in $\rationals(\sqrt{-3}\thinspace)$,
the Euler characteristic $\chi$ of $_{\Gamma_0(\eta)} \backslash \hy$ vanishes~\cite{Vogtmann}, so we only need~$\beta^1$.
Therefore, we indicate $\chi = 1-\beta_1+\beta_2$ only in those Eisenstein integers cases, where it can be non-zero.

\end{appendix}

\bibliographystyle{amsplain}

\end{document}